\documentclass[11pt,a4paper,english]{amsart}
\usepackage{babel}
\usepackage[latin1]{inputenc}
\usepackage{amsmath}
\usepackage{amsthm}
\usepackage{amsfonts}
\usepackage{indentfirst}
\usepackage{graphicx}

\usepackage{amssymb}

\newtheorem{de}{Definition}
\newtheorem{pro}{Proposition}
\newtheorem{cor}{Corollary}
\newtheorem{teo}{Theorem}
\newtheorem{rem}{Remark}
\newtheorem{lem}{Lemma}
\newtheorem{exam}{Example}
\newtheorem{alg}{Algorithm}
\newenvironment{exa}{\begin{exam}\rm}{\cqb\end{exam}}

\newcommand{\gp}{\mathbb{P}}

\newcommand{\mult}{\mathrm{mult}}

\renewcommand{\int}{{\rm int}}

\newcommand{\C}{\ensuremath{\mathbb{C}}}
\newcommand{\CP}{\mathbb{C}\mathbb{P}}
\newcommand{\N}{\ensuremath{\mathbb{N}}}
\newcommand{\Q}{\ensuremath{\mathbb{Q}}}

\newcommand{\X}{\ensuremath{\mathbf{X}}}
\newcommand{\la}{\lambda}

\newcommand{\pd}[2]{\frac{\partial #1}{\partial #2}}

\newcommand{\cqb}{\hbox{}\nobreak\hfill$\Box$}

\newcommand{\DX}{{\mathcal D}({\mathcal X})}
\newcommand{\XL}{\mathcal{X}_{\mathcal L}}

\title[]{A class of polynomial planar vector fields with polynomial first integral}
\author{A.~Ferragut, C.~Galindo \and F.~Monserrat}
\address{A. Ferragut and C. Galindo: Institut de Matem\`atiques i Aplicacions de Castell\'o (IMAC) and Departament de Matem\`{a}tiques, Universitat Jaume I, Edifici TI (ESTEC), Av. de Vicent Sos Baynat, s/n, Campus del Riu Sec, 12071 Castell\'{o} de la Plana, Spain} \email{ferragut@uji.es, galindo@uji.es}
\address{F. Monserrat: E.T.S. d'Inform\`atica Aplicada, Universitat Polit\`ecnica de Val\`encia, Cam\'\i\ de Vera, s/n, 46002 Val\`encia, Spain} \email{framonde@mat.upv.es}
\date{}
\thanks{The first author is partially supported by the Spanish Government grant MTM2013-40998-P. The second and third authors are partially supported by the Spanish Ministry of Economy MTM2012-36917-C03-03 and Universitat Jaume I P1-1B2012-04 grants.}

\begin{document}

\begin{abstract}
We give an algorithm for deciding whether a  planar polynomial differential system has a  first integral which factorizes as a product of defining polynomials of curves with only one place at infinity. In the affirmative case, our algorithm computes a minimal first integral. In addition, we solve the Poincar\'e problem for the class of systems which admit a polynomial first integral as above in the sense that the degree of the minimal first integral can be computed from the reduction of singularities of the corresponding vector field.
\end{abstract}

\maketitle

\section{Introduction}

In this paper we are concerned with planar polynomial differential systems. One of the main open problems in their qualitative theory is to characterize the integrable ones. The importance of the first integral is in its level sets: such a function $H$ whereas it is defined determines the phase portrait of the system, because the level sets $H=h$ give the expression of the solution curves laying on the domain of definition of $H$. Notice that when a differential equation admits a first integral, its study can be reduced in one dimension. In addition, Prelle and Singer \cite{pr-si}, using methods of differential algebra, showed that if a polynomial vector field has an elementary first integral, then it can be computed using Darboux theory of integrability \cite{D}, and Singer \cite{Si} proved  that if it has a Liouvillian first integral, then it has integrating factors given by Darbouxian functions \cite{CL}. Consequently, given a planar differential system, it is important to know whether it has a first integral and compute it if possible. We shall consider complex systems since, even in the real case, invariant curves must be considered over the complex field.

\smallskip

The existence of a rational first integral $H=f/g$ is a very desirable condition for the mentioned systems that guarantees that every invariant curve is algebraic and can be obtained from some equation of type $\lambda f+\mu g=0$, with $(\lambda:\mu)\in\CP^1$, $\CP^1$ being the complex projective line. According to Poincar\'e \cite{poi2}, an element  $(\lambda:\mu)$ is a {\it remarkable} value of $H$ if $\lambda f+\mu g$ is a reducible polynomial in $\C[x,y]$. The curves in its factorization are called {\it remarkable} curves. There are finitely many remarkable values for a given rational first integral $H$ \cite{CGGL2} and the corresponding curves appear to be very important in the phase portrait \cite{FL}. Algebraic integrability has also interest for other reasons. For instance, it is connected with the center problem for quadratic vector fields \cite{sch, cha-lli, lli2, lli} and  with problems related to solutions of Einstein's field equations  in general relativity \cite{hew}.

\smallskip

Prestigious mathematicians as Darboux \cite{dar}, Poincar\'e
\cite{poi1, poi2},  Painlev\'e \cite{pai} and Autonne
\cite{aut} were interested in algebraic integrability. Very interesting problems along this line are the so-called Poincar\'e and Painlev\'e problems. The first one  consists of obtaining an upper bound of the degree $n$ of the first integral depending only on the degree of the polynomial differential system. It is well-known that such a bound does not exist in general \cite{l-n}. However in certain cases a solution is known, for example when the singularities are non-degenerated \cite{poi2}, when the singularities are of nodal type \cite{ce-li} or when the reduction of the system has only one non-invariant exceptional divisor \cite{g-m-4}. Sometimes the problem is stated as bounding the degree $n$ from the knowledge of the system and not only from its  degree. Many other related results are known (including higher dimension)  \cite{car,ca-ca,zam1,soa1,soa2,zam2,pere,es-kl,g-m-1,c-l,g-m-2}. Painlev\'e question, posed in \cite{pai},  asks for recognizing the genus of the general solution of a system as above. Again \cite{l-n} gives a negative answer but, in certain cases and mixing the ideas of Poincar\'e and Painlev\'e, the degree of the first integral can be bounded by using the mentioned genus \cite{g-m-4}.


\smallskip

Darboux gave a lower bound on the number of invariant  integral
algebraic curves of a system as above that ensures the existence of a first integral. A close
result was proved by Jouanolou \cite{jou, CL2} to guarantee that
the system  has a rational first integral and that if
one has enough reduced invariant curves, then the rational first
integral can be computed (see Theorem \ref{tDar}). Furthermore \cite{g-m-1} provides an algorithm to
decide about the existence of  a rational first integral (and to
compute it in the affirmative case) assuming that one has a
well-suited set of $k$ reduced invariant curves, where $k$ is the number of dicritical divisors appearing in the reduction  of the vector field \cite{seid}. Similar results to the above mentioned have been adapted and extended for vector fields in other varieties \cite{jou2,jou, b-n,ghy,cor}.

\smallskip

As a particular case of algebraically integrable systems, one can consider those admitting a polynomial first integral. To the best of our knowledge, there is no characterization for these systems. In this paper, we shall consider the subfamily $\mathfrak{F}$, formed by planar polynomial differential systems with a polynomial first integral which factorizes as a product of curves with only one place at infinity. These curves are a wide class of plane curves characterized by the fact that they meet a certain line (the line at infinity) in a unique point where the curve is reduced and unibranch. They have been rather studied, being \cite{pa1,pa2,am} the most classical papers,  present interesting properties and have been used recently in different contexts \cite{c-p-r-1, c-p-r-2,F-J-ein,F-J-ann,g-m-3}.

\smallskip

We consider the reduction of singularities \cite{seid} of the projective vector field attached to a planar polynomial differential system. This reduction is obtained after finitely many point blowing-ups of the successively obtained vector fields and determines a configuration of infinitely near points of the complex projective plane. Our paper contains two main results. The first one is Corollary \ref{poin}, where we solve the Poincar\'e problem for the polynomial differential systems of the family $\mathfrak{F}$ in the sense that the degree $n$ of the polynomial first integral of a system in $\mathfrak{F}$ can be {\it computed} from its reduction of singularities. In fact, we do not need the complete configuration of infinitely near points as can be seen in the statement. Moreover, $n$ can be {\it bounded} only from the structure (proximity graph) of this reduction. The second main result is an algorithm that decides whether a planar polynomial differential system  belongs to the family $\mathfrak{F}$ and, in the affirmative case, provides a minimal polynomial first integral. We name these first integrals well-behaved at infinity (WAI). The reduction process and certain linear systems related with the above mentioned configuration are our main tools. It is worthwhile to add that our algorithm only performs simple linear algebra computations once the reduction is obtained. The algorithm obtains firstly the irreducible factors of the polynomial first integral and, afterwards, determines the exponents for them. We show two different ways of performing this last step which give rise to what we call Algorithm \ref{alg1} and Algorithm \ref{alg2}.

\smallskip

Our supporting language comes from the algebraic geometry but non-linear ordinary differential equations have interest in practically every science, therefore we feel that it is worthwhile to simplify it as much as possible and provide easy-to-understand explanations for our above mentioned tools. So, Sections \ref{S.settings}, \ref{reduct} and \ref{linsy} are  devoted to provide the reader with information and worked examples on projective vector fields, its reduction procedure and linear systems. This material is not new but we think that, as presented below, it can be read by a wide audience and  will make easy to understand our last section, where our main results are proved.

\smallskip

Section \ref{S:intro} supplies some preliminaries where we define some concepts we shall need, such as  first integral, curve with only one place at infinity, WAI polynomial first integral or projective vector field. Section \ref{sec6} is devoted to explain the intimate relation between planar differential systems which admit a rational first integral and the pencil of curves that this first integral defines. The information we give can be completed in \cite{julio} and is essential for  our main section which is Section \ref{sec7}. Here we state an prove our main theorem, Theorem \ref{char}, whose proof is supported in several previous results given in that section and provides a number of properties that must satisfy a differential system laying in the family $\mathfrak{F}$. These properties are determined by the reduction of singularities of the system and justify Corollary \ref{poin} and Algorithm \ref{alg1}. We conclude by noting that Algorithm \ref{alg2} shows that the before alluded classical results by Darboux and Jouanolou help us to decide about algebraic integrability avoiding the use of some properties of $\mathfrak{F}$. An illustrative example, complementing the mentioned algorithms, is also given at the end of this last section.

\section{WAI polynomial first integrals of planar polynomial vector fields}\label{S:intro}

Along this paper, $\X$ will be the complex planar polynomial differential system given by
\begin{equation}\label{e1}
\dot x=p(x,y),\quad \dot y=q(x,y),
\end{equation}
where $p,q\in\C[x,y]$, $\C$ being the complex field. Let $d=\max\{\deg p,\deg q\}$ be the {\it degree} of  the  system $\X$. We shall also use $\X$  to denote the vector field $\X=p\pd{}x+q\pd{}y$.


A non-constant $\mathcal C^1$-function $H = H(x,y)$ is  a {\it first integral} of $\X$ if $H$ is constant on the solutions of the system. That is, if it satisfies the equation
\[
\X H=p\pd Hx+q\pd Hy=0,
\]
whereas $H$ is defined.

An {\it invariant algebraic curve} of $\X$ is an algebraic curve $C_f$, with local equation $f=0$, $f\in\C[x,y]$, such that
\[
\X f=p\pd fx+q\pd fy=kf,
\]
where $k\in\C[x,y]$. The polynomial $k$ is the {\it cofactor} of $C_f$. It has degree at most $d-1$.


Consider the complex projective plane $\mathbb{C}\mathbb{P}^2$ and homogeneous coordinates $(X:Y:Z)$. Set $L:\{Z=0\}$ the line at infinity. We say that an algebraic curve $C:\{F=0\}$, with $F\in\C[X,Y,Z]$ homogeneous, has {\it only one place at infinity} if $C\cap L$ is a unique point $P$ and $C$ is reduced and unibranch (i.e., analytically irreducible) at $P$. It is easy to find examples of this type of curves and global information for them  can be obtained from local information around its singularity \cite{pa1,pa2,am}.


%
%

In this paper we denote by $\N$ the set of natural numbers $1,2,3,\ldots$. A polynomial function $H(x,y)$ of degree $n\in \N$ is named to be {\it well-behaved at infinity} (WAI for short) if it can be written as
\begin{equation}\label{HPFI}
H=\prod_{i=1}^rf_i^{n_i},
\end{equation}
where $r, n_i\in\N$  and $f_i$ are polynomials in $\C[x,y]$ of degree $d_i\in\N$ such that each curve given by the projectivization $F_i(X,Y,Z)=Z^{d_i}f_i(X/Z,Y/Z)$ of $f_i$ has only one place at infinity.


We shall mainly use the projective version of the system $\X$ into $\CP^2$, thus we shall work with homogeneous coordinates $X,Y,Z$. The vector field $\X$ in these coordinates reads as
\begin{equation} \label{Xproj}
\mathcal X = P \pd{}X + Q \pd{}Y,
\end{equation}
where $P(X,Y,Z)=Z^dp(X/Z,Y/Z)$ and $Q(X,Y,Z)=Z^dq(X/Z,Y/Z)$ are the respective projectivizations of $p$ and $q$. After embedding $\X$ into $\CP^2$, \eqref{HPFI} becomes
\[
\bar H(X,Y,Z)=H(X/Z,Y/Z)=\frac{\prod_{i=1}^rF_i(X,Y,Z)^{n_i}}{Z^n},
\]
where, for each $i$, $F_i(X,Y,Z)$ stands for the projectivization of $f_i$. The main aim of this work is to provide computable steps for discerning whether the system $\X$ has a (minimal) WAI polynomial first integral or not. In the affirmative case, our computations allow us to obtain the mentioned first integral. We recall that a polynomial first integral $H$ of $\X$ is {\it minimal} whenever  any other polynomial first integral has degree at least the degree of $H$.


Later on we shall deal  with singular points of the embedding of our vector field $\X$ into $\CP^2$ and the so-called reduction of its singularities. These concepts are summarized in the following two sections.

\section{Polynomial vector fields in $\CP^2$}\label{S.settings}

Let $A$, $B$, and $C$ be homogeneous polynomials of degree $d+1$ in the complex variables $X$, $Y$, and $Z$. We say that the homogeneous $1$-form
\[
\Omega=AdX + BdY + CdZ
\]
of degree $d+1$ is {\it projective} if $XA+YB+ZC=0$. That is, if there exist three homogeneous polynomials $P$, $Q$, and $R$ of degree $d$ such that
\[
A=ZQ-YR, \; \; B=XR-ZP, \; \; C=YP-XQ.
\]
Then we can write
\begin{equation}\label{omega}
\Omega = P(YdZ - ZdY) + Q(ZdX - XdZ) + R( XdY - YdX).
\end{equation}
Usually in the literature  $\Omega$ is called a Pfaff algebraic form of $\CP^2$; see \cite{jou} for more details. The triple $(P, Q, R)$ can be thought of as a homogeneous polynomial vector field in $\CP^2$ of degree $d$, more specifically
\[
\mathcal X=P\pd{}X +Q\pd{}Y +R\pd{}Z ,
\]
where $X$, $Y$ and $Z$ denote homogeneous coordinates of $\CP^2$.


Let $F\in\C[X,Y, Z]$ be a homogeneous polynomial. The curve  $F=0$ in $\CP^2$  is {\it invariant} under the flow of the vector field $\mathcal X$ if
\begin{equation}\label{XF}
\mathcal XF =P\pd{F}X +Q\pd{F}Y +R\pd{F}Z= KF,
\end{equation}
for some homogeneous polynomial $K \in C[ X, Y, Z]$ of degree $d - 1$, called the {\it cofactor} of $F$.


The singular points of a projective $1$-form $\Omega$ of degree $d+ 1$ or of its associated homogeneous polynomial vector field $\mathcal X$ of degree $d$ are those points satisfying the following system of equations:
\begin{equation}\label{SP}
ZQ-YR=0,\quad XR-ZP=0,\quad YP-XQ=0.
\end{equation}

We devote the remaining of this section to relate affine and projective vector fields. The polynomial differential system \eqref{e1} of degree $d$ is equivalent to the $1$-form
\[
p(x, y) dy-q(x, y) dx,
\]
which can be extended to $\CP^2$ as the projective $1$-form of degree $d+ 1$
\begin{equation}\label{Eq.1form}
Z^{d+2}\left(p\left(\frac XZ,\frac YZ\right)\frac{YdZ-ZdY}{Z^2}-q\left(\frac XZ,\frac YZ\right)\frac{XdZ-ZdX}{Z^2}\right),
\end{equation}
where we have replaced $(x,y)$ by $(X/Z,Y/Z)$. We define $P(X,Y,Z)=Z^dp(X/Z,Y/Z)$ and $Q(X,Y,Z)=Z^dq(X/Z,Y/Z)$. Then \eqref{Eq.1form} becomes
\[
P ( X , Y , Z )( YdZ - ZdY )+ Q ( X , Y , Z) ( ZdX - XdZ ).
\]

In short, the vector field attached to the polynomial differential system \eqref{e1} is extended to  the homogeneous polynomial vector field of degree $d$ in $\CP^2$ $\mathcal X = P \pd{}X + Q \pd{}Y$. This vector field is called the {\it complex projectivization} of  System \eqref{e1} or of the vector field $\X$.


We notice that the third component $R$ in the complex projectivization is identically zero. Consequently the line at infinity $Z = 0$ is a solution of the projective vector field.


From the equalities in \eqref{SP}, we note that the singular points of the complex projectivization of  System \eqref{e1} must satisfy the following equations
\[
ZQ( X, Y, Z) = 0,\quad ZP( X, Y, Z) = 0, \quad YP( X, Y, Z) - XQ( X, Y, Z) = 0. \]
The third equation and the line $Z=0$ determine the singular points at infinity. Setting $Z=1$, the singular points which are not at infinity are obtained from the equality $P=Q=0$.


If $f( x, y) = 0$ is the local equation of an invariant algebraic curve of degree $n\in\N$ of  System \eqref{e1} with cofactor $k( x, y)$, then $F( X, Y, Z) = Z^nf(X/Z,Y/Z)=0$ is an invariant algebraic curve of the vector field  in \eqref{Xproj} with cofactor $K(X,Y, Z)=Z^{d-1}k(X/Z,Y/Z)$.


To end this section we show the behavior of $\mathcal X$ and $K$ when we take local coordinates in the local chart determined by $Z = 1$. The same procedure can be done for $X = 1$ and $Y = 1$. 
Let $F = 0$ be an invariant algebraic curve of degree $n$ of the vector field defined by \eqref{omega} with cofactor $K$. Applying Euler's Theorem for homogeneous functions  and regarding \eqref{XF}, we can prove that $f(x,y)=F (X, Y, 1) = 0$ is an equation of an invariant algebraic curve of the restriction of $\Omega$ to the affine plane:
\[
\left(P(x,y,1)-xR(x,y,1)\right)dy - \left(Q(x,y,1)-yR(x,y,1)\right)dx.
\]
We notice that this $1$-form has degree $d+1$ and the cofactor of $f(x,y)=0$ is $k(x,y) = K(x,y,1)-nR(x,y,1)$. It has degree at most $d$ whenever $Z = 0$ is not invariant. We notice that the line $Z=0$ is invariant if and only if $Z|R$.

\section{Reduction of singularities} \label{reduct}

The main technique to perform the desingularization or the reduction of singular points (of curves or planar vector fields) is the blowing-up (see \cite{C,seid,D,AFJ}). The reduction theorem for planar vector fields was proved by Seidenberg \cite{seid}. Roughly speaking, the blow-up technique transforms, through a change of variables that is not a diffeomorphism, a singularity into a line. Then, for studying the original singular point, one considers the new singular points that appear on this line and that will be, probably, simpler. If some of these new singular points is degenerate, the process is repeated. This iterative process of reduction of singularities is finite. Let us describe it.

\subsection{The blow-up technique}\label{blowup}

Let $M$ be a complex manifold of dimension two. Blowing-up a point $P$ in the manifold $M$ consists on replacing $P$ by a projective line $\mathbb{CP}^1$ considered as the set of limit directions at $P$. Let $T_PM$ be the tangent space of $M$ at $P$ and  $E_P$  the complex projective line given by the projectivization of $T_PM$ with quotient map $[\;]:T_PM\setminus \{0\}\rightarrow E_P$. The {\it blown-up manifold}, denoted by $\mathrm{Bl}_P(M)$, is the set $(M\setminus \{P\})\cup E_P$ endowed with structure of complex manifold of dimension 2 obtained as follows: for each local chart of $M$ at $P$, $(U,\varphi)$, $\varphi=(x,y):U\rightarrow \mathbb{C}^2$, such that $\varphi(P)=(x(P),y(P))=0$,  the pairs $(U_i,\varphi_i)$, $i=1, 2$, will be two local charts of $\mathrm{Bl}_P(M)$ defined as $\varphi_i:V_i^P\rightarrow \mathbb{C}^2$, with
\[
\begin{split}
V_1^P=&(U\setminus x^{-1}(0))\cup (E_P\setminus Ker\;(dx)_P),\\
V_2^P=&(U\setminus y^{-1}(0))\cup (E_P\setminus Ker\;(dy)_P),
\end{split}
\]
and
\[
\begin{split}
\varphi_1=&\left(x,\frac{y}{x}\right)\; \mbox{ in }U\setminus x^{-1}(0) \;\; \mathrm{and} \; \quad \varphi_1\left(\left[\alpha \frac{\partial}{\partial x}+\beta \frac{\partial}{\partial y}\right]\right)=\left(0,\frac{\beta}{\alpha}\right) \; \mathrm{otherwise},\\
\varphi_2=&\left(\frac{x}{y},y\right)\; \mbox{ in } U\setminus y^{-1}(0) \;\; \mathrm{and} \;\quad \varphi_2\left(\left[\alpha \frac{\partial}{\partial x}+\beta \frac{\partial}{\partial y}\right]\right)=\left(\frac{\alpha}{\beta},0\right) \; \mathrm{otherwise}.
\end{split}
\]

The projection map $\pi_P:\mathrm{Bl}_p(M)\rightarrow M$, usually named blow-up of $P$ in $M$, is defined in local coordinates in the following form. If $(x,t=y/x)$ (respectively, $(s=x/y,y)$) are the local coordinates in $V_1^P$ (respectively, $V_2^P$), then $\pi_P(x,t)=(x,xt)$ (respectively, $\pi_P(s,y)=(sy,y)$). The projective line $E_P$ is the \emph{exceptional divisor} of the blow-up and is defined, as a submanifold of $\mathrm{Bl}_P(M)$, by the local equation $x=0$ (respectively, $y=0$) in the chart $(V_1^P,\varphi_1)$ (respectively, $(V_2^P,\varphi_2)$). The restriction of $\pi_P$ to $\mathrm{Bl}_P(M)\setminus E_P$ is a biholomorphism onto $M\setminus \{P\}$. Moreover the equality $\pi_P^{-1}(P)=E_P$ holds.

\subsection{Reduction of singularities}\label{SS.resolution}

Consider the polynomial vector field in $\mathbb{C}^2$ $\X=p\pd{}x+q\pd{}y$. Suppose that it has an isolated singularity at the origin $O$ and consider its associated differential 1-form $\omega=p(x,y)dy-q(x,y)dx$. Let $\omega_{m}=p_m(x,y)dy-q_m(x,y)dx$ be the first non-zero jet of $\omega$ at $O$, where $p_m(x,y)$ and $q_m(x,y)$ are homogeneous polynomials of degree $m$. The integer number $m$ is called the \emph{multiplicity} of $\X$ at $O$.


Consider the blown-up manifold $\mathrm{Bl}_O(\mathbb{C}^2)$,
the projection $\pi_O:\mathrm{Bl}_0(\mathbb{C}^2)\rightarrow \mathbb{C}^2$ and the charts $(V_i^O,\varphi_i)$, $i=1,2$, defined as before. In the chart $(V_1^O,\varphi_1=(x,t))$, we define the \emph{total transform} by $\pi_O$ of the differential 1-form $\omega$ in $V_1^O$ as
\begin{equation}\label{estri1}
\omega^*|_{V_1^O}:=x^m \left[ (\alpha(1,t)+x\beta(x,t))dx+x(p_m(1,t)+x\gamma(x,t))dt \right],\end{equation}
where \begin{equation}
\label{estr}
\alpha(x,y):=yp_m(x,y)-xq_m(x,y)
\end{equation} is the so-called {\it characteristic polynomial} and
\[
\gamma(x,y):=\frac{1}{x^m}\left[p(x,xy)-p_m(x,xy)\right],\quad \beta(x,y):=y\gamma(x,y)-\frac{1}{x^m}\left[q(x,xy)-q_m(x,xy)\right].
\]
The \emph{total transform} by $\pi_O$ of $\omega$ in $V_2^O$ is defined similarly.

\smallskip

Notice that $\omega^*|_{V_1^O}$ is divisible by $x^{m+1}$ if and only if $\alpha(x,y)\equiv 0$. If this holds, we define the \emph{strict transform} by $\pi_O$ of $\omega$ in $V_1^O$ as
$$\tilde{\omega}|_{V_1^O}:=\frac{\omega^*|_{V_1^O}}{x^{m+1}}=\beta(x,t)dx+\left(p_m(1,t)+x \gamma(x,t)\right) dt.$$
Clearly $p_m(x,y)$ is not identically zero in this case and, therefore, at any point of $E_O\cap V_1^O$ where $\beta(x,t)$ does not vanish, the leaves of $\tilde{\omega}|_{V_1^O}$ are transverse to $E_O$. An analogous situation happens for the chart $V_2^O$.


When $\alpha(x,y)\not\equiv 0$, we define the \emph{strict transform} by $\pi_O$ of $\omega$ in $V^O_1$ as $$\tilde{\omega}|_{V_1^O}:=\left(\alpha(1,t)+x \beta(x,t)\right)dx + x \left(p_m(1,t)+x \gamma(x,t)\right)dt.$$
It is easy to deduce that the singular points of $\tilde{\omega}$ that belong to $E_O$ are isolated and moreover that the local curve given by $E_O$ at $O$ is invariant by the vector field defined by $\tilde{\omega}|_{V_1^O}$. As above, we can define $\tilde{\omega}|_{V_2^O}$ is an analogous way.


The differential 1-forms $\tilde{\omega}|_{V_i^O}$, $i=1,2$, define a holomorphic vector field in  $\mathrm{Bl}_O(\mathbb{C}^2)$ denoted by $\tilde{\omega}$. Furthermore, given a holomorphic vector field $\mathcal X$ in any two-dimensional complex manifold $M$ and given any point $P\in M$, restricting to a local chart and applying the above arguments a holomorphic vector field $\tilde{\mathcal X}$ in $\mathrm{Bl}_P(M)$ is defined; we call it the \emph{strict transform} of $\mathcal X$ by $\pi_P$. The above facts give rise to the following definition, which uses the previous notation.

\begin{de}
{\rm Let $O\in \mathbb{C}^2$ be an isolated singularity of a polynomial vector field $\X=p\pd{}x+q\pd{}y$ in $\mathbb{C}^2$. The point $O$ is called a \emph{dicritical} singularity if the polynomial $\alpha$ in (\ref{estr}) is identically zero.  Moreover, $O$ is called a \emph{simple} singularity whenever $\X$ has multiplicity $1$ at $O$ and the matrix
$$\begin{pmatrix} \pd{p_1}x & \pd{p_1}y\\ \pd{q_1}x & \pd{q_1}y \end{pmatrix} $$
has eigenvalues $\lambda_1,\lambda_2$ satisfying either $\lambda_1\lambda_2\not=0$ and $\frac{\lambda_1}{\lambda_2}\not\in\Q^+$, or $\lambda_1\lambda_2=0$ and $\lambda_1^2+\lambda_2^2\neq0$.
Furthermore, an \emph{ordinary singularity} is a singularity that is not simple. We remark that a dicritical singularity is ordinary. Finally, we say that a holomorphic vector field $\mathcal X$ in a two-dimensional complex manifold $M$ has a {dicritical}  (respectively, simple, ordinary) singularity at $P\in M$ if its restriction to a local chart at $P$ has a dicritical (respectively, simple, ordinary) singularity at the corresponding point in $\mathbb{C}^2$.}
\end{de}

By Equality (\ref{estri1}), the following characterization of non-dicritical singularities holds:

\begin{pro}\label{nondicritical}
A singularity $P$ of a holomorphic vector field $\mathcal X$ in a two-dimensional complex manifold $M$ is non-dicritical if and only if the exceptional divisor of the blown-up manifold ${\rm Bl}_P(M)$ is invariant by the strict transform of $\mathcal X$ in ${\rm Bl}_P(M)$.
\end{pro}

Generically speaking, simple singularities $P$ of holomorphic vector fields $\mathcal{X}$ cannot be reduced by blow-ups, that is, the strict transform of $\mathcal X$ in $\mathrm{Bl}_P(M)$, where $P$ is a simple singularity, may have  simple singularities at the points of the exceptional divisor $E_P$. By a classical result of Seidenberg \cite{seid} (see also \cite{Brunella} for a modern treatment) the remaining singularities of such vector fields can be eliminated or reduced to simple ones:

\begin{teo}
Let $\mathcal X$ be a holomorphic vector field in a two-dimensional complex manifold $M$ with isolated singularities. Then there exists a finite sequence of blow-ups such that the strict transform of $\mathcal X$ in the last obtained complex manifold has no ordinary singularities.
\end{teo}

Let $P$ be a point in a two-dimensional complex manifold $M$. The {exceptional divisor} $E_P$  produced by blowing up $P$ is called the {\it first infinitesimal neighborhood} of $P$. By induction, if $i>0$, then the points in the $i$-th infinitesimal neighborhood of $P$ are the points in the first infinitesimal neighborhood of some point in the $(i-1)$-th infinitesimal neighborhood of $P$. A point $Q$ in some infinitesimal neighborhood of $P$ is called to be \emph{proximate} to $P$ if $Q$ belongs to the strict transform of $E_P$ (see Section \ref{linsy} for a definition of strict transform of a curve). Also $Q$ is a \emph{satellite} point if it is proximate to two points; that is, if it is the intersection point of the strict transforms of two exceptional divisors. Non-satellite points are named \emph{free}.


Points in the $i$-th infinitesimal neighborhood of $P$, for some $i>0$, are said to be {\it infinitely near} to $P$. These points admit a natural ordering that we shall use in this paper and call ``to be infinitely near to", where a point $R$ precedes $Q$ if and only if $Q$ is infinitely near to $R$. Note that we agree that a point is infinitely near to itself.


A \emph{configuration of infinitely near points} of $M$ (or, simply, a configuration) is a finite set
\[
{\mathcal C}=\{Q_0, \ldots , Q_n\},
\]
such that $Q_0 \in {X}_0 = M$ and $ Q_i \in \mathrm{Bl}_{Q_{i-1}}( X_{i-1}) =: X_{i} \stackrel{\pi_{Q_{i-1}}}{\longrightarrow} X_{i-1}$,
 for $1 \leq i \leq n$;
where we have denoted by $\mathrm{Bl}_{Q_{i-1}} (X_{i-1})$ the blown-up manifold corresponding to blow-up $Q_{i-1}$ in $X_{i-1}$.


The Hasse diagram of $\mathcal C$ with respect to the above alluded order relation is a union of rooted trees  whose set of vertices is bijective with $\mathcal C$. We join with a dotted edge those vertices corresponding with points $P$ and $Q$ of $\mathcal C$ such that $Q$ is proximate to $P$ but $Q$ is not in the first infinitesimal neighborhood of $P$. The obtained labeled graph, denoted $\Gamma_{\mathcal C}$, is called the {\it proximity graph} of $\mathcal{C}$.

Example \ref{Ex.1} below shows the reduction of a singular point of a vector field and its proximity graph.

\begin{de}
{\rm The  \emph{singular configuration} of a holomorphic vector field $\mathcal X$ in a two-dimen\-sional complex manifold $M$, denoted by $\mathcal S(\mathcal X)$, is the union
${\mathcal S}(\mathcal X):=\cup_P {\mathcal S}_P(\mathcal X)$,
where $P$ runs over the set of ordinary singularities of $\mathcal X$ and ${\mathcal S}_P(\mathcal X)$ denotes the set of points $Q$ infinitely near to $P$ such that the strict transform of $\mathcal X$ has an ordinary singularity at $Q$.  The proximity graph $\Gamma_{{\mathcal S}(\mathcal X)}$ is called the \emph{singular graph} of $\mathcal X$.}
\end{de}

\begin{de}
{\rm Let $\mathcal X$ be a holomorphic vector field  in a two-dimen\-sional complex manifold $M$. The \emph{dicritical configuration} of $\mathcal X$ is  the set $\mathcal D(\mathcal X)$ of points $P\in {\mathcal S}(\mathcal X)$ such that there exists a point $Q\in {\mathcal S}(\mathcal X)$ that is infinitely near to $P$ and is a dicritical singularity of the strict transform of ${\mathcal X}$ in the blown-up manifold to which $Q$ belongs. These dicritical singularities $Q$ in $\DX$ will be called {\it infinitely near dicritical singularities of ${\mathcal X}$}.}
\end{de}

\begin{exa}\label{Ex.1} {\rm
Consider the homogeneous polynomial vector field $\mathcal X$ in $\mathbb{CP}^2$ defined by $$2XZ^4\;dX+5Y^4Z\;dY-\left(5Y^5+2X^2Z^3\right)dZ.$$ Its singularities are the points $P=(1:0:0)$ and $Q=(0:0:1)$.


Take affine coordinates $y=\frac{Y}{X}$ and $z=\frac{Z}{X}$ in the chart defined by $X\not=0$, where the point $P$ has coordinates $(y,z)=(0,0)$. The differential form in these coordinates is
$\omega_1:=5y^4z\;dy-(5y^5+2z^3)\;dz$.
$\mathcal X$ has an ordinary singularity at $P$. Consider the blow-up $\pi_P: X_1:=\mathrm{Bl}_P(\mathbb{CP}^2)\rightarrow \mathbb{CP}^2$ and coordinates $(y_1=y, z_1=z/y)$ in the chart $V_1^P$. Then, the strict transform of $\omega$ in $V_1^P$ is
\[
\tilde{\omega}_1|_{V_1^P}=-2z_1^4\; dy_1-(5y_1^3+2y_1z_1^3)\; dz_1.
\]
The unique ordinary singularity of the vector field defined by $\tilde{\omega}_1|_{V_1^P}$ is
$P_1:=(y_1,z_1)=(0,0)$. It belongs to the exceptional divisor $E_P$, whose local equation is $y_1=0$. Moreover, taking local coordinates in the chart $V_2^P$, it is easy to see that the unique point of $E_P$ that is not in $V_1^P$ is not a singularity of $\tilde{\mathcal X}$.


Now we consider the blow-up $\pi_{P_1}: X_2:=\mathrm{Bl}_{P_1}(X_1)\rightarrow X_1$ and affine coordinates $(y_2=y_1, z_2=z_1/y_1)$ in the chart $V_1^{P_1}$. The strict transform of $\omega_1$ in $V_1^{P_1}$ is
$$\tilde{\omega}_1|_{V_1^{P_1}}=(-5z_2-4y_2z_2^4)\; dy_2+(-5y_2-2y_2^2z_2^3)\; dz_2.$$
The unique singularity in $E_{P_1}\cap V_1^{P_1}$ of the strict transform of $\mathcal X$ is $P'_2:=(0,0)$; it is straightforward to check that it is a simple singularity.


Taking coordinates $(y_2=y_1/z_1, z_2=z_1)$ in $V_2^{P_1}$, we get
$$\tilde{\omega}_1|_{V_2^{P_1}}=-2z_2^2\; dy_2+\left(-5y_2^3-4y_2z_2\right)\; dz_2.$$  Then, the strict transform of $\mathcal X$ has an ordinary singularity at the unique point $P_2\in E_{P_1}\setminus V_1^{P_1}$, whose coordinates in $V_2^{P_1}$ are $(0,0)$. Since the local equation of the strict transform of $E_{P}$ in $V_2^{P_1}$ is $y_2=0$, it holds  $\{P_2\}=E_{P_1}\cap E_{P}$ and, therefore, $P_2$ is a satellite point that is proximate to $P_1$ and $P$.


Next, we have to perform the blow-up $\pi_{P_2}: X_3:=\mathrm{Bl}_{P_2}(X_2)\rightarrow X_2$ and
$$\tilde{\omega}_1|_{V_1^{P_2'}}=\left(-5y_3z_3-6z_3^2\right)\; dy_3+\left(-5y_3^2-4y_3z_3\right)\; dz_3,$$ in local coordinates $(y_3=y_2, z_3=z_2/y_2)$. The unique singularity of the strict transform of $\mathcal X$ in $E_{P_2}\cap V_1^{P_2}$ is $P_3:=(0,0)$, that belongs to the strict transform of $E_{P_1}\cap E_{P_2}$ (notice that the local equation of $E_{P_1}$ in $V_1^{P_2}$ is $z_3=0$). It is an ordinary singularity.  It is straightforward to verify that the unique point in $E_{P_2}\setminus V_1^{P_2}$ is a simple singularity.


Considering now the blow-up $\pi_{P_3}: X_4:=\mathrm{Bl}_{P_3}(X_3)\rightarrow X_3$ and local coordinates $(y_4=y_3, z_4=z_3/y_3)$ at $V_1^{P_3}$ we have that
\[
\tilde{\omega}_1|_{V_1^{P_3}}=(-10z_4-10z_4^2)\; dy_4+(-5y_4-4y_4z_4)\; dz_4.
\]
There are two new singularities at $E_{P_3}\cap V_1^{P_3}$ which are $R:=(0,0)$ and $P_4=(0,-1)$. The point $R$ is a simple singularity and, applying the change of coordinates $y_4'=y_4,\;\; z_4'=z_4+1$, it holds that
$$\tilde{\omega}_1|_{V_1^{P_3}}=\left(10z_4'-10z_4'^2\right)\; dy_4'+\left(-y_4'-4y_4'z_4'\right)\; dz_4',$$ and therefore $P_4$ is an ordinary singularity. Moreover it is easy to check that the unique point in $E_{P_3}\setminus V_1^{P_3}$ is a simple singularity.


Now, for $i\in \{4,5,\ldots, 12\}$ we consider the blow-up $\pi_{P_i}: X_{i+1}:=\mathrm{Bl}_{P_i}(X_i)\rightarrow X_i$, the coordinates $(y_{i+1}':=y_i', z_{i+1}':=z_i'/y_i')$ at $V_1^{P_i}$ and $P_{i+1}:=(0,0)\in E_{P_i}\cap V_1^{P_i}$. It is easy to check that the strict transform of $\mathcal X$ in $X_{i+1}$ has multiplicity $1$ at $P_{i+1}$. Its unique singularity in $E_{P_i}$ is $P_{i+1}$. It is ordinary, and non-dicritical whenever $i\leq 11$. Moreover
$$\tilde{\omega}_1|_{V_1^{P_{12}}}=[z_{13}'-42(y_{13}')^9(z_{13}')^2]\; dy_{13}'+[-y_{13}'-4(y_{13}')^{10}z_{13}']\; dz_{13}',$$
and, then, $P_{13}$ is a dicritical singular point. The strict transform of $\mathcal X$ in $X_{13}$ has not ordinary singularities in $E_{P_{13}}$.



Now we consider coordinates $x=\frac{X}{Z}$ and $y=\frac{Y}{Z}$ in the chart defined by $Z\not=0$, where the point $Q$ has coordinates $(x,y)=(0,0)$. The differential form that defines the restriction of $\mathcal X$ is
$$\omega_2:=2x\;dy+5y^4\;dy.$$
$Q$ is an ordinary singularity of $\mathcal X$ and its reduction process  is described in Table \ref{tabla}. The first column indicates the chart where each point (proper or infinitely near) of ${\mathcal S}_Q(\mathcal X)$ is located. The second  column corresponds to the system of local coordinates that we consider and the corresponding points. The last column shows the differential 1-forms that define the strict transforms of $\mathcal X$ at every point. Notice that $Q_3$ belongs to the strict transform of $E_{Q_1}$ and therefore $Q_3$ is proximate to $Q_1$. Observe also that $Q$, $Q_1$, $Q_2$ and $Q_3$ are non-dicritical points.

\begin{table}[t!]
\centering
    \begin{tabular}{||c|c|c||}
  \hline \hline Chart & System of coordinates & Differential form\\\hline \hline
    $Z\not=0$ & $(x=X/Z, y=Y/Z)$ at $Q$& $2x\;dx + 5y^4\; dy$ \\
    $V_2^{Q}$ & $(x_1=x/y, y_1=y)$ at $Q_1$ & $2x_1y_1\;dx_1+(2x_1^2+5y_1^3)\; dy_1$\\
    $V_2^{Q_1}$ & $(x_2=x_1/y_1, y_2=y_1)$ at $Q_2$ & $2x_2y_2\; dx_2+(4x_2^2+5y_2)\; dy_2$\\
    $V_1^{Q_2}$ & $(x_3=x_2, y_3=y_2/x_2)$  at $Q_3$ & $(6x_3y_3+5y_3^2)\; dx_3+(4x_3^2+5x_3y_3)\; dy_3$\\ \hline \hline
    \end{tabular}
    \label{tabla}
    \vspace{0.4cm}

\caption{Reduction of the singularity at $Q$.}
\end{table}

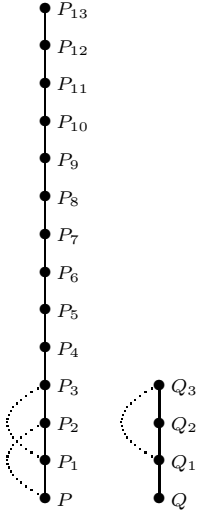
\begin{figure}[h!]
\centering
\setlength{\unitlength}{0.5cm}
\begin{picture}(9,14)

\qbezier[20](3,0)(1,1)(3,2)
\qbezier[20](3,1)(1,2)(3,3)

\put(3,0){\circle*{0.3}}\put(3.3,-0.2){\tiny{$P$}}
\put(3,0){\line(0,1){1}}

\put(3,1){\circle*{0.3}}\put(3.3,0.8){\tiny{$P_1$}}
\put(3,1){\line(0,1){1}}

\put(3,2){\circle*{0.3}}\put(3.3,1.8){\tiny{$P_2$}}
\put(3,2){\line(0,1){1}}

\put(3,3){\circle*{0.3}}\put(3.3,2.8){\tiny{$P_3$}}
\put(3,3){\line(0,1){1}}

\put(3,4){\circle*{0.3}}\put(3.3,3.8){\tiny{$P_4$}}
\put(3,4){\line(0,1){1}}

\put(3,5){\circle*{0.3}}\put(3.3,4.8){\tiny{$P_5$}}
\put(3,5){\line(0,1){1}}

\put(3,6){\circle*{0.3}}\put(3.3,5.8){\tiny{$P_6$}}
\put(3,6){\line(0,1){1}}

\put(3,7){\circle*{0.3}}\put(3.3,6.8){\tiny{$P_7$}}
\put(3,7){\line(0,1){1}}

\put(3,8){\circle*{0.3}}\put(3.3,7.8){\tiny{$P_8$}}
\put(3,8){\line(0,1){1}}

\put(3,9){\circle*{0.3}}\put(3.3,8.8){\tiny{$P_9$}}
\put(3,9){\line(0,1){1}}

\put(3,10){\circle*{0.3}}\put(3.3,9.8){\tiny{$P_{10}$}}
\put(3,10){\line(0,1){1}}

\put(3,11){\circle*{0.3}}\put(3.3,10.8){\tiny{$P_{11}$}}
\put(3,11){\line(0,1){1}}

\put(3,12){\circle*{0.3}}\put(3.3,11.8){\tiny{$P_{12}$}}
\put(3,12){\line(0,1){1}}

\put(3,13){\circle*{0.3}}\put(3.3,12.8){\tiny{$P_{13}$}}


\put(6,0){\circle*{0.3}}\put(6.3,-0.2){\tiny{$Q$}}
\put(6,0){\line(0,1){1}}

\put(6,1){\circle*{0.3}}\put(6.3,0.8){\tiny{$Q_1$}}
\put(6,1){\line(0,1){1}}

\put(6,2){\circle*{0.3}}\put(6.3,1.8){\tiny{$Q_2$}}
\put(6,2){\line(0,1){1}}

\put(6,3){\circle*{0.3}}\put(6.3,2.8){\tiny{$Q_3$}}

\qbezier[20](6,1)(4,2)(6,3)

\end{picture}
\caption{Proximity graph of $\mathcal{S}(\mathcal{X})$.}
\end{figure}
}
With the above notation, we have  ${\mathcal S}(\mathcal X)=\{P,Q\}\cup \{P_i\}_{i=1}^{13}\cup \{Q_i\}_{i=1}^3$ and ${\mathcal D}(\mathcal X)=\{P\}\cup \{P_i\}_{i=1}^{13}$. Figure 1 shows the proximity graph of the configuration ${\mathcal S}(\mathcal X)$.
\end{exa}

\section{Linear systems} \label{linsy}

\subsection{Linear systems associated with clusters}

Along this section we consider the complex projective plane $\mathbb{CP}^2$ and fix homogeneous coordinates $X, Y, Z$.

\begin{de}
{\rm
A \emph{linear system} on $\mathbb{CP}^2$ is the set of algebraic curves given by a linear subspace of $\mathbb{C}_m[X,Y,Z]\cup \{0\}$ for some natural number $m>0$, where $\mathbb{C}_m[X,Y,Z]$ denotes the set of homogeneous polynomials of degree $m$ in the variables $X, Y, Z$. If the dimension (as a projective space) of a linear system is $1$, then it is called a {\it pencil}.}
\end{de}


\begin{de}  {\rm
A \emph{cluster} of infinitely near points (or, simply,  a cluster) of $\mathbb{CP}^2$ is a pair $({\mathcal C}, {\mathbf m})$ where $\mathcal C=(Q_0,\ldots,Q_h)$ is a configuration of infinitely near points of $\mathbb{CP}^2$ and ${\mathbf m}=(m_0,\ldots,m_h)\in \N^n$.}
\end{de}

Our next step is to define linear systems on $\mathbb{CP}^2$ given by a pair formed by a cluster ${\mathcal C}$ and a positive integer. To this purpose, for each $Q_i\in {\mathcal C}$, let us denote by  $\ell(Q_i)$  the cardinality of the set $\{Q_j\in {\mathcal C} | \mbox{ $Q_i$ is infinitely near to $Q_j$} \}$.

\begin{de}\label{virtual} {\rm
Consider a cluster ${\mathcal K}=({\mathcal C}, {\mathbf m})$, an algebraic curve $C$ in $\mathbb{CP}^2$, and a point $Q_k\in {\mathcal C}$.  Assume $\ell(Q_k)=1$, that is $Q_k$ is only infinitely near to itself.  Take a local chart at $Q_k$ with local coordinates $(x,y)$ and let $f(x,y)=0$ be a local equation of $C$. We define the \emph{virtual transform} of $C$ at $Q_k$ with respect to the cluster $\mathcal K$ (denoted by $C^{\mathcal K}_{Q_k}$) as the (local) curve defined by $f(x,y)=0$. Moreover we say that $C$ \emph{passes virtually} through $Q_k$ with respect to $\mathcal K$ if the multiplicity of $C^{\mathcal K}_{Q_k}$ at $Q_k$ (that is, the degree of the first non-zero jet of $f(x,y)$), denoted by $m_{Q_k}(C^{\mathcal K}_{Q_k})$, is greater than or equal to $m_k$.

\smallskip

Suppose now that $\ell(Q_k)>1$. Let $Q_j\in {\mathcal C}$ be such that $Q_k$ is in the first infinitesimal neighborhood  of $Q_j$ and assume inductively that $C$ \emph{passes virtually} through $Q_j$ with respect to $\mathcal K$. Take local coordinates $(x,y)$ at $Q_j$ and let $f(x,y)=0$ be a local equation of $C^{\mathcal K}_{Q_j}$. We can write  $Q_k=(0,\lambda)\in V_1^{Q_j}$ (respectively, $Q_k=(\lambda,0)\in V_2^{Q_j}$) in local coordinates $(x,t=y/x)$ (respectively, $(s=x/y,y)$). Then we define the \emph{virtual transform} of $C$ at $Q_k$ with respect to the cluster $\mathcal K$ as the (local) curve defined by $x^{-m_j}f\left( x,x(t+\lambda) \right)=0$ (respectively, $x^{-m_j}f\left((s+\lambda)y,y\right)=0$). We denote it by $C^{\mathcal K}_{Q_k}$.  The above equations define also what we  call virtual transform (centered at $Q_k$)  of $C$ at the chart $V_1^{Q_j}$ (respectively, $V_2^{Q_j}$). Moreover, we say that $C$ \emph{passes virtually} through $Q_k$ with respect to $\mathcal K$ if the multiplicity of $C^{\mathcal K}_{Q_k}$ at $Q_k$, denoted by $m_{Q_k}(C^{\mathcal K}_{Q_k})$, is greater than or equal to $m_k$. Finally, the curve $C$ \emph{passes virtually} through $\mathcal K$ if it passes virtually through $Q_i$ with respect to $\mathcal K$ for all $Q_i\in {\mathcal K}$.

}

\end{de}

The {\it strict transform} $\tilde{C}$ of an algebraic curve $C$ in a manifold obtained by a sequence of point blowing-ups is the global curve given by the virtual transform through the cluster of points and multiplicities defined by the curve. Note the analogy with the similar definition given in Section \ref{SS.resolution}.

\begin{de}
{\rm
Given a positive integer $m$ and a cluster $\mathcal K=({\mathcal C},\mathbf{m})$ of $\mathbb{CP}^2$, the \emph{linear system determined by $m$ and $\mathcal K$}, denoted by
${\mathcal L}_m(\mathcal K)$ or ${\mathcal L}_m({\mathcal C},\mathbf{m})$, is the linear system on $\mathbb{CP}^2$ given by those curves defined by polynomials in $\mathbb{C}_m[X,Y,Z]\cup \{0\}$ that pass virtually through $\mathcal K$.}
\end{de}

\begin{exa}
Consider the points $P=(0:0:1)$ and $Q=(1:0:1)$ of $\mathbb{CP}^2$, whose coordinates in the chart defined by $Z\not=0$ are $(x=\frac{X}{Z}=0, y=\frac{Y}{Z}=0)$ and $(x=1,y=0)$, respectively. Consider also the following infinitely near to $P$ points: $P_1=(0,3)\in V_1^{P}$ and $P_2=(1,0)\in V_2^{P_1}$, with the  notations of Section \ref{blowup}.


Consider the cluster ${\mathcal K}=({\mathcal C}, \mathbf{m})$, where ${\mathcal C}=\{Q, P, P_1, P_2\}$ and $\mathbf{m}=(2,2,1,1)$. Let us compute the linear system ${\mathcal L}_3(\mathcal K)$. To do that, consider an arbitrary projective curve $C\in {\mathcal L}_3(\mathcal K)$ defined by an homogeneous polynomial of degree $3$ with undetermined coefficients:
$$aX^3+bX^2Y+cX^2Z+dXY^2+eXYZ+fXZ^2+gY^3+hY^2Z+iYZ^2+kZ^3,$$
whose expression in the chart $Z\not=0$ is
$$ax^3+bx^2y+cx^2+dxy^2+exy+fx+gy^3+hy^2+iy+k.$$
On the one hand, since the multiplicity of $C$ at $P$ must be greater than or equal to 2, it follows that $f=i=k=0$. On the other hand, the multiplicity of $C$ at $Q$ must be greater than or equal to 2, so the coefficients of the monomials of degree less than 2 of
$$c (1 + x)^2 + a (1 + x)^3 + e (1 + x) y +b (1 + x)^2 y + h y^2 +
 d (1 + x) y^2 + g y^3$$
are equal to 0; that is, $a=c=0$ and $b=-e$.


 The local equation defining the virtual transform of $C$ at $P_1$, $C^{\mathcal K}_{P_1}$, is
$$3e+9 h + (9 d -3e + 27 g )x_1 + (e + 6 h) y_1 + (6 d - e + 27 g )x_1 y_1 +
 h  y_1^2 + (d  + 9 g )x_1 y_1^2 + g x_1 y_1^3=0$$
 in the coordinates $(x_1=x, y_1=y/x)$.
Therefore, since the multiplicity of $C^{\mathcal K}_{P_1}$ at $P_1$ must be greater than or equal to 1, we get $e=-3h$.
Finally, the local equation of the virtual transform of $C$ at $P_2$ with respect to $\mathcal K$ is
$$3 h+ (9 d  + 27 g + 9h) x_2 + h y_2+ (6 d + 27 g + 3h) x_2 y_2 + (d  +
 9 g) x_2 y_2^2 + g x_2 y_2^3=0,$$
where $x_2=x_1/y_1$ and $y_2=y_1$. Thus $C^{\mathcal K}_{P_2}$ passes virtually through $P_2$ with respect to $\mathcal K$ if and only if $h=0$.


As a consequence, ${\mathcal L}_3(\mathcal K)$ is the projective space generated by curves given by the monomials  $XY^2$ and $Y^3$; that is, the curves in ${\mathcal L}_3(\mathcal K)$ are those defined by an equation of the type $Y^2L=0$, where $L=\alpha X+\beta Y$, for some $(\alpha, \beta) \in \mathbb{C}^2\setminus \{(0,0)\}$.
\end{exa}

\subsection{Cluster of base points of a linear system}\label{bp}

Let $n$ be a positive integer and  $\mathcal L$  a linear system on $\mathbb{CP}^2$ such that ${\mathcal L}$ is given by $\mathbb{P}V$, where $V=\langle F_1, F_2, \ldots,F_s\rangle$ is the linear space over $\mathbb{C}$ spanned by linearly independent polynomials $F_1, F_2, \ldots,F_s\in \mathbb{C}_n[X,Y,Z]$. Assume that $F_1, F_2,\ldots,F_s$ have no common factor. Then, there exists a configuration of (infinitely near) points of $\mathbb{CP}^2$, $\mathcal{BP}({\mathcal L})$, and a finite set of linear subspaces ${\mathcal H}_i\subsetneq \mathbb{CP}^{s-1}$, $1\leq i\leq t$, such that the strict transforms of the curves with equations $$\alpha_1 F_1(X,Y,Z)+ \alpha_2 F_2(X,Y,Z)+ \cdots+\alpha_s F_s(X,Y,Z)=0,$$  $(\alpha_1, \alpha_2, \ldots,\alpha_s)\in \mathbb{CP}^{s-1}\setminus \bigcup_{i=1}^t {\mathcal H}_i$ (which, in the sequel, we call \emph{generic} curves of $\mathcal L$) have the same multiplicities at every point  $Q\in \mathcal{BP}({\mathcal L})$ (denoted by $\mathrm{mult}_Q({\mathcal L})$) and have empty intersection at the manifold obtained by blowing-up the points in $\mathcal{BP}({\mathcal L})$.  Notice that, if ${\mathcal L}$ is a pencil, then $\bigcup_{i=1}^t {\mathcal H}_i$ is a finite set.

\begin{de} {\rm
 The cluster $(\mathcal{BP}({\mathcal L}),\mathbf{m})$, with $\mathcal{BP}({\mathcal L})$ as it was defined above and $\mathbf{m}=\left(\mathrm{mult}_Q({\mathcal L})\right)_{Q\in \mathcal{BP}({\mathcal L})}$, is the  \emph{cluster of base points} of $\mathcal L$.}
\end{de}

\begin{exa}
Let $\mathcal L$ be the linear system on $\mathbb{CP}^2$ defined by the curves $\alpha F(X,Y,Z)+\beta Z^5=0$, where $F(X,Y,Z):=X^2Z^3+Y^5$ and $(\alpha,\beta)\in\mathbb{C}^2\setminus \{(0,0)\}$. It is easy to check that the configuration of base points $\mathcal{BP}({\mathcal L})$ coincides with the configuration ${\mathcal D}({\mathcal X})$ of Example \ref{Ex.1}.

\begin{table}[ht!]
\centering
    \begin{tabular}{||c|c|c||}
   \hline \hline Chart & System of coordinates & Strict transform of a generic curve\\ \hline \hline
    $X\not=0$ & $(y=Y/X, z=Z/X)$ at $P$&  $\alpha (z^3+y^5)+\beta z^5$ \\
    $V_1^{P}$ & $(y_1=y, z_1=z/y)$ at $P_1$&  $\alpha (y_1^2+z_1^3)+\beta y_1^2z_1^5$ \\
    $V_2^{P_1}$ & $(y_2=y_1/z_1, z_2=z_1)$ at $P_2$ & $\alpha(z_2+y_2^2)+\beta y_2^2z_2^5 $\\
    $V_1^{P_2}$ & $(y_3=y_2, z_3=z_2/y_2)$ at $P_3$& $\alpha(y_3+z_3)+\beta y_3^6z_3^5 $ \\
    $V_1^{P_3}$ & $(y_4=y_3, z_4=z_3/y_3+1)$ at $P_4$ & $\alpha z_4+\beta y_4^{10}(z_4-1)^5$\\
    $V_1^{P_{i-1}}$ & $(y_i=y_{i-1}, z_i=z_{i-1}/y_{i-1})$ at $P_i$ & $\alpha z_i+\beta y_i^{14-i}(z_iy_i^{i-4}-1)^5$\\
     \hline\hline
    \end{tabular}
\vspace{0.4cm}
\label{tabla2}
\caption{Base points of $\mathcal L$. We note that $5 \leq i \leq 13$  and $\alpha\neq0$.}
\end{table}

Table 2
shows the local expressions of the successive strict transforms of the generic elements of the linear system. Then the cluster of base points of $\mathcal L$ is $\left({\mathcal D}({\mathcal X}),(3,2,1_{12})\right)$, where $1_{12}$ means a sequence of 12 ones.
\end{exa}

\section{Resolution of a pencil and infinitely near dicritical points}

\label{sec6}

In this section, we shall  briefly describe the resolution process of a pencil of curves in $\mathbb{CP}^2$ and compare it with the reduction of singularities of the vector field ${\mathcal X}$ whose invariant curves are given by the pencil (that is, the quotient of  two different curves of the pencil provides a rational first integral of ${\mathcal X}$). Additional information can be found in \cite{julio}.

Consider a pencil ${\mathcal L}$ given by $\mathbb{P}\langle F_1,F_2\rangle$, where $F_1, F_2$ are polynomials in ${\mathbb C}_n[X,Y,Z]$ (for some positive integer $n$) without common components. Let $P$ be any point in $\mathcal{BP}({\mathcal L})$. As in Definition \ref{virtual}, take local coordinates $(x,y)$ at $P$ and consider the virtual transforms of the elements in ${\mathcal L}$ with respect to the cluster $\left({\mathcal C},(m_Q)_{Q \in \mathcal C}\right)$, where ${\mathcal C}:=\{Q\in \mathcal{BP}({\mathcal L})\mid Q\not=P\mbox{ and } Q \mbox{ is infinitely near to } P\}$ and $m_Q:=\mult_Q({\mathcal L})$ for every $Q$. These virtual transforms will be given by polynomials
\begin{multline*}
\alpha f_1(x,y)+\beta f_2(x,y)\\
=D(x,y)\left(\alpha f_1^{(r)}(x,y)+\beta f_2^{(r)}(x,y)\right)+\alpha f_1^{(>m_P)}(x,y)+\beta f_2^{(>m_P)}(x,y),
\end{multline*}
where $m_P:=\mult_P({\mathcal L})$, $f_i^{(j)}$ (respectively, $f_i^{(>j)}$) denotes the $j$-th jet of $f_i$ (respectively, $f_i-f_i^{(j)}$), $i=1,2$, $j \in {\mathbb N}$,  $D(x,y)$ is the greatest common divisor of $f_1^{(m_P)}$ and $f_2^{(m_P)}$, and $r:=m_P-d$,  where $d=\deg(D)$. Notice that,  except for finitely many elements $(\alpha: \beta)\in \mathbb{CP}^1$, the above expression defines the strict transform of a generic element of ${\mathcal L}$. The virtual transforms in the chart $V_1^{P}$ (with local coordinates $(x_1:=x, y_1:=y/x)$) of the elements in ${\mathcal L}$ on the manifold obtained after blowing-up $P$ are defined by
\begin{multline}\label{v2}
D(1,y_1)\left(\alpha f_1^{(r)}(1,y_1)+\beta f_2^{(r)}(1,y_1)\right)\\
+x_1\left(\alpha f_1^{(m_P+1)}(1,y_1)+\beta f_2^{(m_P+1)}(1,y_1)+\cdots\right).
\end{multline}
A similar expression is obtained in the chart $V_2^P$. The points in $\mathcal{BP}({\mathcal L})\cap V_1^P$ have the form $(0,\xi)$, $\xi$ being a root of the polynomial $D(1,t)$.

\begin{de}
{\rm With the above notations, a point $P$ in $\mathcal{BP}({\mathcal L})$ is said to be \emph{dicritical} with respect to ${\mathcal L}$ if $r>0$.

}
\end{de}

\begin{rem}\label{nota1}
{\rm
From the expression (\ref{v2}), it is clear that $P$ is dicritical whenever it is a maximal point of $\mathcal{BP}({\mathcal L})$ with respect to the ordering ``to be infinitely near to" (because $D(x,y)=1$ in this case).
}
\end{rem}

Let $X$ be the manifold obtained after blowing-up the points in $\mathcal{BP}({\mathcal L})$ and let $P\in X$. Let $S$ be that point of $\mathcal{BP}({\mathcal L})\cap \mathbb{CP}^2$ such that $P$ is proximate to $S$. Assume without loss of generality that $S=(0:0:1)$. Performing changes of coordinates in the successive blowing-ups as described in Section \ref{blowup}, we  obtain a system of coordinates $(x,y)$ at $P$ and polynomials $g_1(x,y), g_2(x,y)$ such that $\alpha g_1(x,y)+\beta g_2(x,y)=0$, $(\alpha: \beta)\in \mathbb{CP}^1$, are the equations at $P=(0,0)$ of the virtual transforms of the elements in ${\mathcal L}$ with respect to the cluster of base points of ${\mathcal L}$.  Notice that $g_1$ and $g_2$ do not vanish simultaneously at $(0,0)$.

As a consequence of the above paragraph, the assignment $P\mapsto (g_1(0,0):g_2(0,0))$ defines a holomorphic map $\varphi:X\rightarrow \mathbb{CP}^1$ that extends to $X$ the rational map $$\phi: \mathbb{CP}^2\cdots \rightarrow \mathbb{CP}^1$$ given by $\phi:  S \mapsto (F_1(S):F_2(S))$ (eliminating its indeterminacies); that is, $\phi\circ \pi=\varphi$, where $\pi: X \rightarrow \mathbb{CP}^2$ is the before alluded composition of blowing-ups.

\begin{pro}\label{proposition2}
With the above notations, consider a point $P\in \mathcal{BP}({\mathcal L})$. The following statements are equivalent:
\begin{itemize}
\item[(a)] $P$ is not dicritical with respect to ${\mathcal L}$.

\item[(b)] The strict transform on $X$ of the exceptional divisor $E_P$, also denoted $E_P$, is a component of the virtual transform of some curve in ${\mathcal L}$ with respect to the cluster of base points of $\mathcal L$.

\item[(c)] $E_P$ is a component of some fiber of the holomorphic map $\varphi:X\rightarrow \mathbb{CP}^1$ that the pair $(F_1,F_2)$ defines.

\item[(d)] $\mult_P({\mathcal L}) = \sum_Q \mult_Q({\mathcal L})$, where the sum is taken over the set of  proximate to $P$ points in $\DX$.
 \end{itemize}

\end{pro}

\begin{proof}
$P$ is not a dicritical point with respect to ${\mathcal L}$ if and only if $(f_1^{(r)}(1,y_1), f_2^{(r)}(1,y_1))=(a,b)\in \mathbb{C}^2\setminus \{(0,0)\}$. By Equality (\ref{v2}), this happens if and only if $E_P$ is a component of the virtual transform (with respect to the cluster of base points of ${\mathcal L}$) of the curve defined by $bF_1(X,Y,Z)-aF_2(X,Y,Z)=0$. This shows the equivalence between (a) and (b).

The equivalence between (b) and (c) is clear because the fibers of $\varphi$ are just the curves in $X$ defined by the virtual transforms of the elements in ${\mathcal L}$ with respect to the cluster of base points of ${\mathcal L}$.

To end the proof, we can assume (performing a change of variables if necessary) that $x$ does not divide $D(x,y)$. Then $P$ is non-dicritical  with respect to ${\mathcal L}$ if and only if $D(1,y_1)=\prod_{i=1}^{q}(y_1-\xi_i)^{d_i}$, where $q, d_i\in \mathbb{N}$, $\xi_i\in \mathbb{C}$, $1\leq i\leq q$, $\xi_i\not=\xi_j$ if $i\neq j$, and $\sum_{i=1}^{q} d_i=m_P$. This is equivalent to say that the strict transform of a generic curve of $\mathcal L$ meets $E_P$ at $q$ different points $R_i$ (with local coordinates $(0,\xi_i)$), $1\leq i\leq q$, and $m_P=\sum_{i=1}^{q} d_i$, where $d_i$ is the intersection multiplicity at $R_i$ of the just mentioned strict transform  and $E_P$. Taking  into account that the points of $\mathcal{BP}(\mathcal{L})$ belonging to the intersection of the strict transforms of a generic curve and $E_P$ are proximate to $P$, it holds that the equivalence between (a) and (d) follows from Noether Formula \cite[Theorem 3.3.1]{C}, which is showed later in (\ref{eq.Noether}).
\end{proof}


For a pencil $\mathcal L$ as at the beginning of the section, consider the vector field $\mathcal{X}_{\mathcal L}$ in $\mathbb{CP}^2$ whose invariant curves are given by the pencil. This vector field is defined by the homogeneous 1-form (in projective coordinates) $\Omega_{\mathcal{L}}:=AdX+BdY+ZdZ$, where  $(A,B,C)=(A',B',C')/\gcd(A',B',C')$ and
$$A':=F_2\frac{\partial F_1}{\partial X}-F_1\frac{\partial F_2}{\partial X},\;\; B':=F_2\frac{\partial F_1}{\partial Y}-F_1\frac{\partial F_2}{\partial Y},\;\; C':=F_2\frac{\partial F_1}{\partial Z}-F_1\frac{\partial F_2}{\partial Z}.$$

Now set $x,y$ local coordinates at an open neighborhood $V$ of a point $P$ in a two-dimensional complex manifold $M$, and $f,g$  holomorphic functions in $V$. Consider the \emph{local pencil} $\Gamma$ of curves in $V$ defined by equations $\alpha f+\beta g=0$, where $(\alpha:\beta)$ runs over $\mathbb{CP}^1$. Its {\it associated vector field in $V$} is defined by the 1-form $\omega_{\Gamma}:=a(x,y)dx+b(x,y)dy$, where $(a(x,y), b(x,y)):=(\bar{a}(x,y), \bar{b}(x,y))/\gcd(\bar{a},\bar{b})$ and $\bar{a}(x,y)=g\frac{\partial f}{\partial x}-f\frac{\partial g}{\partial x}$, $\bar{b}(x,y):=g\frac{\partial f}{\partial y}-f\frac{\partial g}{\partial y}$. It is not difficult to verify that the local vector fields defined by the pencils given by the restrictions of $F_1$ and $F_2$ to the corresponding affine charts patch together to give rise to the \emph{global} vector field $\mathcal{X}_{\mathcal{L}}$.








\begin{lem}
\label{el1uno}
With the above notations, let $\Gamma$ be a \emph{local} pencil at a point $P\in M$. Then, the operations on $\Gamma$ ``blowing-up'' and ``taking associated 1-forms'' commute. More specifically,  let $\pi$ the blow-up of $P$ in $M$ and consider strict transforms with respect to $\pi$. If $\tilde{\Gamma}$ is the \emph{local} pencil at an open neighborhood of $Q\in E_P$ spanned by the strict transforms of two generic elements of $\Gamma$, then $\omega_{\tilde{\Gamma}}=\tilde{\omega}_{\Gamma}$, where $\tilde{\omega}_{\Gamma}$ denotes the strict transform of $\omega_{\Gamma}$.
\end{lem}

\begin{proof}
Assume that $f$ and $g$ are generic elements of $\Gamma$. Take local coordinates $x', y'$ at $V_1^P$.
On the one hand, it holds
$$\omega_{\tilde{\Gamma}}=\frac{ \bar{\omega}_{\tilde{\Gamma}}}{\gcd({a}',{b}')},$$ where
$\bar{\omega}_{\tilde{\Gamma}}={a}'(x',y')dx+{b}'(x',y')dy$,
${a}'(x',y')=\tilde{g}\frac{\partial \tilde{f}}{\partial x'}-\tilde{f}\frac{\partial \tilde{g}}{\partial x'}$ and ${b}'(x',y'):=\tilde{g}\frac{\partial \tilde{f}}{\partial y'}-\tilde{f}\frac{\partial \tilde{g}}{\partial y'}$, $\tilde{f}$ and $\tilde{g}$ being the strict transforms of  $f$ and $g$ at $V_P^1$. On the other hand,  the strict transform of $\omega_{\Gamma}$ in $V_P^1$ is
$$\tilde{\omega}_{\Gamma}=\frac{\omega^*_{\Gamma}}{\gcd(\bar{a},\bar{b})},$$
where
\[
\begin{split}
\bar{a}(x',y'):=& \; g(x',x'y')\frac{\partial f}{\partial x}(x',x'y')-f(x',x'y')\frac{\partial g}{\partial x}(x',x'y')\\
&+y'\left(g(x',x'y')\frac{\partial f}{\partial y}(x',x'y')-f(x',x'y')\frac{\partial g}{\partial y}(x',x'y')\right),\\
\bar{b}(x',y'):=& \; g(x',x'y')\frac{\partial f}{\partial y}(x',x'y')-f(x',x'y')\frac{\partial g}{\partial y}(x',x'y'),\\
\omega_{\Gamma}^*=& \; \bar{a}(x',y')dx'+\bar{b}(x',y')dy'.
\end{split}
\]

Let $h(x,y)$ be a polynomial whose multiplicity at $(0,0)$ is $m$ and write $h(x',x'y')=(x')^m\tilde{h}(x',y')$. The following identities hold: $\frac{\partial h}{\partial x}(x',x'y')=\frac{\partial (h(x',x'y'))}{\partial x'}-y'\frac{\partial h}{\partial y}(x',x'y')$ and $\frac{\partial h}{\partial y}(x',x'y')=x'\frac{\partial (\tilde{h}(x',x'y'))}{\partial y'}$. Setting $s$  the multiplicity of the curves defined by $f$ and $g$ at $P$, the above identities allow us to prove that $\omega_{\Gamma}^*=(x')^{2s} \bar{\omega}_{\tilde{\Gamma}}$ and so our result holds since it suffices to take reduced forms.
\end{proof}

\begin{pro}
Let $\mathcal{L}$ be a pencil as at the beginning of this section and let $P$ be a base point of $\mathcal L$. Then, $P$ is dicritical with respect to $\mathcal L$ if and only if $P$ is an infinitely near dicritical singularity of $\XL$.
\end{pro}

\begin{proof}
Let $f$ be a polynomial in the local variables $x,y$ at $P$ defining the strict transform at $P$ of a generic element of $\mathcal L$. Let $m$ be the multiplicity of $f$ at $P$.

Assume that $P$ is dicritical with respect to $\mathcal{L}$ and take a polynomial $g$ defining the strict transform of an element of $\mathcal L$ different from that given by $f$. Then the initial forms $f^{(m)}$ and $g^{(m)}$ of $f$ and $g$ are linearly independent. Consider the local vector field $\omega_{\mathcal{L}}$ that the pencil determines at $P$  as defined above Lemma \ref{el1uno}.
Set $$\bar{a}(x,y)dx+\bar{b}(x,y)dy=h(x,y)\cdot \omega_{\mathcal{L}}$$ and $h^{(i)}$ the initial form of $h$. Following the notations of Section \ref{SS.resolution}, we get
\[
\begin{split}
h^{(i)}(x,y)\alpha(x,y)=& \; y\left(\frac{\partial f^{(m)}}{\partial y}g^{(m)}-f^{(m)}\frac{\partial g^{(m)}}{\partial y}\right)+x\left(\frac{\partial f^{(m)}}{\partial x}g^{(m)}-f^{(m)}\frac{\partial g^{(m)}}{\partial x}\right)\\
=& \left( x \frac{\partial f^{(m)}}{\partial x}+y\frac{\partial f^{(m)}}{\partial y}\right)g^{(m)}- \left( x \frac{\partial g^{(m)}}{\partial x}+y\frac{\partial g^{(m)}}{\partial y}\right)f^{(m)}\\
=& \; mf^{(m)}g^{(m)}-mf^{(m)}g^{(m)}=0,
\end{split}
\]
which, by Lemma \ref{el1uno}, proves that $P$ is an infinitely near dicritical singularity of $\XL$.

To finish our proof, suppose that $P$ is not dicritical with respect to $\mathcal L$. Then there exists an element of $\mathcal L$ whose strict transform at $P$ is defined by an equation $g(x,y)=0$ such that the multiplicity of $g$ at $P$ is $n>m$. Now, by repeating the same computation as before, it happens that
$$h^{(i)}(x,y) \alpha(x,y)=mf^{(m)}g^{(n)}-nf^{(m)}g^{(n)}=(m-n)f^{(m)}g^{(n)}\not=0.$$
Hence $P$ is not an infinitely near dicritical singularity of $\XL$.
\end{proof}

As a consequence of the above proposition, the following result holds.
\begin{cor}\label{corollary1}
Let $\mathcal L$ be a pencil given by two homogeneous polynomials  of the same degree without common components. Then $\mathcal{BP}(\mathcal L)=\mathcal{D}(\XL)$.
\end{cor}


\section{Main results}\label{sec7}

\subsection{The main theorem. Poincar\'e problem and Algorithm \ref{alg1}}
In this section, unless otherwise stated, we shall assume that the vector field $\X$ has a WAI polynomial first integral and, as before, we shall denote by $\mathcal X$ the complex projectivization of $\X$. The existence of a WAI polynomial first integral implies that of a minimal one $H$, that will be what we always consider.



Keep the notations as in Section \ref{S:intro}. The rational function $\bar H$ is an equivalent datum to the pencil ${\mathcal P}_{\mathcal X}:=\mathbb{P}\langle F_1^{n_1} F_2^{n_2} \cdots F_r^{n_r},Z^n\rangle$
and, by \cite[Lemma 1]{g-m-1},
${\mathcal L}_n(\mathcal{BP}_{\mathcal X})={\mathcal P}_{\mathcal X}$,
where $\mathcal{BP}_{\mathcal X}$ denotes the cluster of base points of ${\mathcal P}_{\mathcal X}$. This means that one can compute the first integral $H$ from the integer number $n$ and the cluster $\mathcal{BP}_{\mathcal X}$. We shall show that the dicritical configuration ${\mathcal D}({\mathcal X})$ determines both data.

Next theorem is our first step. To prove it we shall use the B\'ezout-Noether Formula (see \cite[Corollary I.7.8]{hart} and \cite[Theorem 3.3.1]{C}) which, for two algebraic curves $C_1$  and $C_2$  on $\CP^2$, states that
\begin{equation}\label{eq.Noether}
\deg C_1\deg C_2=\sum_Q I_Q(C_1,C_2)=\sum_Pm_P(\tilde C_1)m_P(\tilde C_2),
\end{equation}
where $\tilde C_1$ and $\tilde C_2$ stand for the strict transforms of $C_1$ and $C_2$ in some manifold obtained by blowing-up, $Q$ (respectively, $P$) runs over the set $C_1\cap C_2$ (respectively, of infinitely near points to some $Q$ as above, $P$, such that $P\in \tilde C_1\cap\tilde C_2$) and $I_Q(C_1,C_2)$ denotes the intersection multiplicity at $Q$ of $C_1$ and $C_2$. In addition, we consider a system of multiplicities $\mathbf{m}({\mathcal C}, \mathcal{C}')$ attached with any pair of configurations of infinitely near points $\mathcal C$ and $\mathcal C'$ of $\mathbb{CP}^2$ such that $\mathcal{C}\subseteq \mathcal{C}'$. This is defined as $\mathbf{m}({\mathcal C}, \mathcal{C}'):=(m_Q)_{Q\in {\mathcal C'}}$, where $m_Q=1$  if $Q$ is a maximal point of $\mathcal C$, $m_Q=0$ if $Q\in \mathcal{C}'\setminus \mathcal{C}$ and $m_Q=\sum_P m_P$ otherwise, the sum running over the set of points $P\in {\mathcal C}$ such that $P$ is proximate to $Q$. Finally, set
\[
\mathrm{Fr}({\mathcal C}):=\{P\in {\mathcal C}| \mbox{ $P$ is a free point}\}
\]
and, for each $P\in {\mathcal C}$, define
$${\mathcal C}^P:=\{Q\in {\mathcal C} | \mbox{ $P$ is infinitely near to $Q$}\}.$$

\begin{teo}\label{prrr}
With the notations as in Section \ref{S:intro}, let ${\bf X}$ be a polynomial vector field having a WAI polynomial first integral $H=\prod_{i=1}^rf_i^{n_i}$ and  $\mathcal{X}$ its complex projectivization. Then:
\begin{itemize}

\item[(1)] The configurations of infinitely near points ${\mathcal D}({\mathcal X})$ and $\mathcal{BP}({\mathcal P}_{\mathcal X})$ coincide.

\item[(2)] $\DX$ has exactly $r$ maximal points with respect to the ordering ``to be infinitely near to", which we denote by $R_1, R_2, \ldots,R_r$. Moreover these maximal points are the unique infinitely near dicritical singularities  of ${\mathcal X}$.

\item[(3)] The set $\mathrm{Fr}({\mathcal D}({\mathcal X}))$ has exactly $r$ maximal elements and, for each $i\in \{1, 2, \ldots,r\}$, each point $R_i$ is infinitely near to one of these maximal elements, which we denote by $M_i$.

\item[(4)] For each $i\in \{1, 2, \ldots,r\}$, set $\mathbf{m}\big({\mathcal D}({\mathcal X})^{M_i},\DX\big)=(h^i_Q)$ the above defined system of multiplicities. Then, up to reordering of $\{1, 2, \ldots,r\}$, ${\mathcal D}({\mathcal X})^{M_i}$ is the set of points in ${\mathcal D}({\mathcal X})$ through which the strict transforms of the curve $C_i$, defined by $F_i=0$, pass. Moreover, for all $Q\in {\mathcal D}({\mathcal X})^{M_i}$, it holds that $\mathrm{mult}_Q(\tilde{C}_i)=h^i_Q$ and the degrees $d_i$ of the curves $C_i$ satisfy
\begin{equation}\label{inter}
d_i=\sum_{Q\in {\mathcal D}({\mathcal X})^{M_i}\cap \tilde{L}} h^i_Q,
\end{equation}
 where
${\mathcal D}({\mathcal X})^{M_i}\cap \tilde{L}$ is the set of points in ${\mathcal D}({\mathcal X})^{M_i}$ through which the strict transforms of the line of infinity pass.
\end{itemize}
\end{teo}

\begin{proof}
Statement (1) follows fom Corollary \ref{corollary1}.
We claim that the fact that we consider $H$ minimal proves the following statements:
\begin{itemize}
\item[(1)] $\gcd(n_1, n_2, \ldots,n_r)=1$.
\item[(2)] Either $r=1$ (and $n_1=1$), or $r\geq 2$ and there exists $i\in \{2, 3, \ldots,r\}$ such that $f_i-f_1\not\in \mathbb{C}$.
\end{itemize}
Indeed,  $\delta : =\gcd(n_1, n_2, \ldots,n_r)\neq 1$ implies that $H^{1/\delta}$ is also a first integral, which is a contradiction with the mentioned minimality of the first integral. To show (2), assume that $r\geq 2$ and, for all $i\in \{2, 3, \ldots,r\}$, $f_i= f_1 + \alpha_i$ for some $\alpha_i\in \mathbb{C}$, then  $H=T(f_1)$, where $T(t):=t^{n_1}\prod_{i=2}^r (t+\alpha_i)^{n_i}$; so $f_1$ is a first integral, which is also a contradiction.

Now consider the pencils  ${\mathcal P}_i:=\mathbb{P}\langle F_i, Z^{d_i}\rangle$, $1 \leq i \leq r$. From a careful reading of the statement and proof of  \cite[Lemma 1]{c-p-r-2}, we deduce the following facts:


(i) Each configuration $\mathcal{BP}({\mathcal P}_i)$ is contained into $\mathcal{BP}({\mathcal P}_{\mathcal X})$ and has exactly $1$ maximal point, which we denote by $N_i$. Moreover $N_i\not=N_j$ for $i \neq j$.


(ii) $\mathcal{BP}({\mathcal P}_{\mathcal X})=\cup_{i=1}^r {\mathcal C}_i$, where ${\mathcal C}_i=\mathcal{BP}({\mathcal P}_i)\cup \{Q_{i,1}, Q_{i,2}, \ldots, Q_{i,k_i}\}$, $Q_{i,1}$ belongs to the first infinitesimal neighborhood of $N_i$ and $Q_{i,j}$ belongs to the first infinitesimal neighborhood of $Q_{i,j-1}$ for $2 \leq j \leq k_i$.

(iii) The maximal point with respect to the proximity relation of ${\mathcal C}_i$, $1 \leq i \leq r$,   through which the strict transform of $C_i$ passes is the maximal free point of ${\mathcal C}_i$ (that we denote by $M_i$).
\smallskip

(iv) Let $\pi:X\rightarrow \mathbb{CP}^2$ be the composition of the blow-ups of the points of the configuration $\DX$ and let $\phi :\mathbb{CP}^2\cdots \rightarrow \mathbb{CP}^1$ be the rational map defined by $\bar H$ (see the paragraph above Proposition \ref{proposition2}). The exceptional divisors $E_P$ (with $P\in \DX$) are mapped by $\varphi=\phi\circ \pi$ to a point of $\mathbb{CP}^1$ with the exception of the divisors in the set $\{E_{Q_{i,k_i}}\}_{i=1}^r$, whose images are $\mathbb{CP}^1$.



For $1\leq i\leq r$, the composition of the blow-ups of the points in $\mathcal{BP}({\mathcal P}_i)$ provides an embedded resolution of the branch of $C_i$ at infinity and the strict transform of $C_i$ passes through $N_i$. Therefore  $M_i\in \{Q_{i,1}, Q_{i,2}, \ldots Q_{i,k_i}\}$ by our above assertion (iii). This implies that $M_i\neq M_j$ if $\neq j$. Then it is clear that the set of points in ${\mathcal D}({\mathcal X})$ through which the strict transforms of the curve $C_i$ pass is ${\mathcal D}({\mathcal X})^{M_i}$, with multiplicity $h^i_Q$ for all $Q\in{\mathcal D}({\mathcal X})^{M_i}$ and so (3) and the first statement in (4) are proved.

On the one hand, defining $R_i:=Q_{i,k_i}$, $1 \leq i \leq r$, it holds that $R_1, R_2, \ldots,R_r$ are the maximal elements of $\DX$. On the other hand, $\{\varphi^{-1}(\lambda)\}_{\lambda\in \mathbb{CP}^1}$ is the set of invariant curves of the strict transform of $\mathcal{X}$ at the manifold $X$ obtained after blowing-up the points in $\DX$ (see \cite{julio}, for instance). This means, by (iv), that the unique exceptional divisors in $X$ that are not invariant by the strict transform of $\mathcal{X}$ are $E_{R_i}$, $1\leq i\leq r$. Then, by Proposition \ref{nondicritical}, the points $R_i$, $1\leq i\leq r$, are the unique infinitely near dicritical singularities of $\mathcal X$. This proves (2).

Finally B\'ezout-Noether Formula for the curves $C_i$ and the line at infinity proves Equality (\ref{inter}), which concludes our proof.
\end{proof}









We next introduce some equalities that will be useful later on. For $\X$ as in Theorem \ref{prrr} and with the same notation, set $r_P:=m_P({\mathcal P}_{\mathcal X})$, for $P \in {\mathcal D}({\mathcal X})$. The first equation below follows from B\'ezout-Noether Formula \eqref{eq.Noether} for two generic curves of ${\mathcal P}_{\mathcal X}$. It relates the degree $n$ of the curves in ${\mathcal P}_{\mathcal X}$ (that is, the degree of the rational first integral of $\mathcal X$) and the multiplicities $r_P$ above defined:
\begin{equation}\label{Eq.0}
n^2=\sum_{P\in{\mathcal D}({\mathcal X})}r_P^2.
\end{equation}
The same formula with respect to a generic curve of ${\mathcal P}_{\mathcal X}$ and $C_i$, $1 \leq i \leq r$, gives rise to
\begin{equation}\label{Eq.1}
n \; d_i=\sum_{P\in{\mathcal D}({\mathcal X})}h_P^i \cdot r_P.
\end{equation}
Applying again the same formula \eqref{eq.Noether} to a generic curve of ${\mathcal P}_{\mathcal X}$ and the line of infinity $L$, we get
\begin{equation}\label{Eq.2}
n=\sum_{P\in{\mathcal D}({\mathcal X})\cap\tilde L}r_P.
\end{equation}

Finally, let us define ${\mathcal N}({\mathcal X})$ as the set of non-maximal points of the dicritical configuration $\DX$.

For any $Q\in {\mathcal N}({\mathcal X})$ and as a consequence of Item (2) of Theorem \ref{prrr} and Proposition \ref{proposition2}, we have
\begin{equation}\label{Eq.3}
r_Q=\sum_P r_P,
\end{equation}
where the sum runs over the points $P$ in $\DX$ which are proximate to $Q$.
\smallskip

By \cite[Lemma 1]{c-p-r-2} it holds that the strict transform of a generic element of the pencil ${\mathcal P}_{\mathcal X}$ at each free maximal point $M_i$ has a local equation of the type $\alpha u^{a_i}+\beta t^{\ell_i}$, where $u=0$ (respectively, $t=0$) is a local equation of the strict transform of $C_i$ at $M_i$ (respectively, the exceptional divisor), $a_i$ and $\ell_i$ being natural numbers. Then, straightforward computations involving Equality (\ref{Eq.3}) show that $r_{R_i}=\gcd(a_i,\ell_i)$ and, as a consequence, the following result happens.

\begin{lem}\label{bbb}
The greatest common divisor $\gcd(\{r_P\mid P\in \DX\})$  equals one.
\end{lem}

Let $N$ be the cardinality of ${\mathcal D}({\mathcal X})$. We introduce the non-degenerated symmetric bilinear pairing over the vector space $\mathbb{R}^{N+1}$,   $\langle \cdot \rangle: \mathbb{R}^{N+1}\times \mathbb{R}^{N+1} \rightarrow \mathbb{R}$ such that if ${\bf a}=(a_0; (a_P)_{P\in \DX}),{\bf b}=(b_0; (b_P)_{P\in \DX})\in \mathbb{R}^{N+1}$, then
\begin{equation}
\label{bilineal}
\langle {\bf a}, {\bf b}\rangle :=a_0b_0-\sum_{P\in \DX} a_Pb_P.
\end{equation}
For $P\in \DX$, set
$${\bf e}_P:=(0; (m_Q^P)_{Q\in{\mathcal D}({\mathcal X})}),$$
where $m_Q^P$ equals $-1$ (respectively, $1$, $0$) if $Q=P$ (respectively, $Q$ is proximate to $P$, otherwise). It is not difficult to check that
\[
\langle {\bf e}_P, {\bf e}_P\rangle<0\; \mbox{ and }\;\langle {\bf e}_P, {\bf e}_Q\rangle\in \{0,1\} \;\mbox{for all $P,Q\in \DX$ such that $P\not=Q$}.
\]
In addition, equalities \eqref{Eq.1} and \eqref{Eq.3} mean that the vector $\left(n; (r_P)_{P\in {\mathcal D}({\mathcal X})}\right)\in \mathbb{R}^{N+1}$ belongs to the orthogonal complement (with respect to the above defined bilinear pair) of the subspace of $\mathbb{R}^{N+1}$ spanned by the set
\begin{equation}
\label{ctoese}
S:=\left\{{\bf c}_i:=(d_i; (h^i_P)_{P\in{\mathcal D}({\mathcal X})})\right\}_{i=1}^r\cup \left\{{\bf e}_Q\right\}_{Q\in{\mathcal N}({\mathcal X})}.
\end{equation}
Notice that the cardinality of $S$ is $N$.

In the sequel and for any tuple $\mathbf{m}=(m_0,(m_P)_{P\in \DX})\in \mathbb{N}^{N+1}$, we shall write $\mathcal{L}(\mathbf{m})$ instead of $\mathcal{L}_{m_0}(\DX,(m_P)_{P\in \DX})$. Using this notation, we state the following result:

\begin{lem}\label{lema3}
Let $Q\in \DX$ and  $\mathbf{m}=\left(m_0,(m_P)_{P\in \DX}\right)\in \mathbb{N}^{N+1}$.
Then $\mathcal{L}(\mathbf{m})\subseteq \mathcal{L}(\mathbf{m}+\mathbf{e}_Q)$.
\end{lem}

\begin{proof}
Consider the clusters $\mathcal{K}:=\left(\DX,(m_P)_{P\in \DX}\right)$ and  $\mathcal{K}':=\left(\DX,(m'_P)_{P\in \DX}\right)$, where $m'_P=m_P-1$ if $P=Q$; $m'_P=m_P+1$ if $P$ is proximate to $Q$; and $m'_P=m_P$ otherwise. Let $x,y$ be local coordinates at a point $T\in \DX$ in the first infinitesimal neighborhood of $Q$  and let $f(x,y)=0$ be the local equation of the virtual transform at $T$ of a curve $C$ in $\mathcal{L}(\mathbf{m})$ with respect to the cluster $\mathcal{K}$. Then, the virtual transform at $T$ of $C$ with respect to $\mathcal{K}'$ is $xf(x,y)$, where $x=0$ is assumed to be the equation of $E_Q$. Moreover it is clear that the new factor $x$ increases in one unit the multiplicity of the virtual transform at any point proximate to $Q$ and different from $T$. Therefore $C$ belongs to $\mathcal{L}(\mathbf{m}+\mathbf{e}_Q)$.
\end{proof}

With notations as before, set
\begin{equation}
\label{elr}
\mathbf{r} := \left(  n; (r_P)_{P\in \DX}   \right).
\end{equation}

The following properties are key facts for our main results. The first one is \cite[Lemma 1]{g-m-1} and is stated without proof.

\begin{lem}\label{lema4}
$\mathcal{L}({\mathbf r})={\mathcal P}_{\mathcal X}$.
\end{lem}

\begin{lem}\label{agh}
Let $C$ be a curve in $\mathbb{CP}^2$. Then, $C$ is invariant by $\mathcal X$ if and only if $\langle \mathbf{r}, \mathbf{c}\rangle=0$, where $\mathbf{c}=(d := \deg C; (\mult_P(\tilde{C}))_{P\in \DX})$.

\end{lem}

\begin{proof}
Without loss of generality we can assume that $C$ is reduced and irreducible. Let $\pi:X\rightarrow \mathbb{CP}^2$ be the composition of blowing-ups of the points in $\mathcal{BP}(\mathcal{P}_{\mathcal{X}})$. Statement (1) of Theorem \ref{prrr} shows that $\mathcal{BP}(\mathcal{P}_{\mathcal{X}})=\DX$. So, $C$ is an invariant curve of $\mathcal{X}$ if and only if it is a component of some curve in the pencil $\mathcal{P}_{\mathcal{X}}$, that is, if and only if the strict transform $\tilde{C}$ on $X$ does not meet the strict transform of a generic curve $D$ of the pencil (see the paragraph below Remark \ref{nota1}). This concludes our statement because it is equivalent to B\'ezout-Noether Formula for the curves $C$ and $D$ over the points in $\DX$.
\end{proof}

\begin{lem}\label{independent}
The  set $S \subseteq \mathbb{R}^{N+1}$ defined in (\ref{ctoese}) is linearly independent.
\end{lem}

\begin{proof}

Reasoning by contradiction, assume that $S$ is linearly dependent. This means that there exist two disjoint subsets $I_1$ and $I_2$ of the set $\{1, 2, \ldots,r\}$, two disjoint subsets $J_1$ and $J_2$ of the set ${\mathcal N}({\mathcal X})$ and positive integers $\alpha_i,\beta_Q$, $i\in I_1\cup I_2$, $Q\in J_1\cup J_2$ such that
\begin{equation}\label{Eq.5}
\sum_{i\in I_1} \alpha_i {\bf c}_i+\sum_{Q\in J_1} \beta_Q {\bf e}_Q=\sum_{i\in I_2} \alpha_i {\bf c}_i+\sum_{Q\in J_2} \beta_Q {\bf e}_Q.
\end{equation}

Taking coordinates $(x_0; (x_P)_{P\in \DX})$, Equality (\ref{Eq.0}) proves that the vector ${\bf r}$ defined in (\ref{elr}) spans a generatrix $\mathfrak{G}$ of the cone $\mathfrak{C}$ of $\mathbb{R}^{N+1}$ defined by the equation $\sum_{P\in \DX}x_P^2=x_0^2$. Moreover, $nx_0-\sum_{P\in \DX} r_Px_P = 0$ is an equation of the hyperplane $\mathfrak{H}$ tangent to $\mathfrak{C}$ which contains $\mathfrak{G}$. Equations (\ref{Eq.1}) and (\ref{Eq.3}) show that $S$ is contained in $\mathfrak{H}$ and, therefore, $\sum_{P\in \DX} y_P^2\geq y_0$ for any ${\bf y}=(y_0; (y_P)_{P\in \DX})$ in the span of $S$. In addition, the  equality happens if and only if ${\bf y}$ belongs to $\mathfrak{G}$. In other words, $\langle {\bf y},{\bf y}\rangle \leq 0$ for every ${\bf y}$ belonging to the span of $S$, and equality holds if and only if ${\bf y}$ is a multiple of the vector ${\bf r}$.

Let ${\bf d}$ be the vector given by the left (or the right) hand side of Equality (\ref{Eq.5}). The above paragraph shows that $\langle {\bf d}, {\bf d}\rangle\leq 0$. Moreover, from Equality (\ref{Eq.5}) we deduce that
\begin{multline*}
\left\langle {\bf d}, {\bf d}\right\rangle=\left\langle \sum_{i\in I_1} \alpha_i {\bf c}_i, \sum_{i\in I_2} \alpha_i {\bf c}_i \right\rangle+\left\langle \sum_{i\in I_1} \alpha_i {\bf c}_i, \sum_{Q\in J_2} \beta_Q {\bf e}_Q   \right\rangle\\
+\left\langle \sum_{i\in I_2} \alpha_i {\bf c}_i,  \sum_{Q\in J_1} \beta_Q {\bf e}_Q  \right\rangle+\left\langle \sum_{Q\in J_1} \beta_Q {\bf e}_Q,  \sum_{Q\in J_2} \beta_Q {\bf e}_Q  \right\rangle,
\end{multline*}
which allows us to deduce that
\begin{equation}\label{eq1}
\langle {\bf d}, {\bf d}\rangle=0.
\end{equation}
Indeed, this is a consequence of the following inequalities that hold for $1 \leq i , j \leq r$, $i \neq j$ and $P,Q\in \DX$, $P\not=Q$: $\langle {\bf c}_i,{\bf c}_j\rangle=d_id_j-\sum_{Q\in \DX} h_Q^ih_Q^j\geq 0$ which happens by B\'ezout-Noether Formula; $\langle {\bf c}_i,{\bf e}_P \rangle=h_P^i-\sum_{Q} h_Q^i\geq 0$, where $Q$ runs over the set of proximate to $P$ points in $\DX$ \cite[Theorem 4.2.2]{C}; and $\langle {\bf e}_P,{\bf e}_Q\rangle\geq 0$.

As a consequence of Equality (\ref{eq1}), ${\bf d}$ is a multiple of ${\bf r}$ and therefore, following the notations of Lemma \ref{lema3}, $\mathcal{L}({\bf d})=\mathcal{L}(\nu{\bf r})$ for some positive integer $\nu$. Applying Lemma \ref{lema3} to both sides of Equality (\ref{Eq.5}), it holds that
$$\mathcal{L}\left(\sum_{i\in I_1} \alpha_i {\bf c}_i\right)\subseteq \mathcal{L}(\nu{\bf r})\;\;\mbox{ and }\;\; \mathcal{L}\left(\sum_{i\in I_2} \alpha_i {\bf c}_i\right)\subseteq \mathcal{L}(\nu{\bf r}).$$
In particular, the curves $D_1$ and $D_2$ defined, respectively, by $H_1:=\prod_{i\in I_1} F_i^{\alpha_i}=0$ and $H_2:=\prod_{i\in I_2} F_i^{\alpha_i}=0$,
belong to the linear system $\mathcal{L}(\nu{\bf r})$.


Let ${\mathcal G}$ be the set of monomials of degree $\nu$ in two variables, $T_1$ and $T_2$, and consider the linear system ${\mathcal T}$ spanned by the set $$\{G(H_1,H_2)\mid G\in {\mathcal G}\}.$$
Recall that $\mathcal{L}(\bf{r})={\mathcal P}_{\mathcal X}$ by Lemma \ref{lema4}. It is clear that a curve defined by an equation $F(X,Y,Z)=0$ belongs to ${\mathcal T}$ if and only if $F=G_1 G_2 \cdots G_{\nu}$, where each $G_i(X,Y,Z)=0$ defines a curve in the pencil ${\mathcal P}_{\mathcal X}=\mathcal{L}(\bf{r})$. To end our proof, we shall prove that $\mathcal{T} = \mathcal{L}(\nu{\bf r})$, which provides the desired contradiction because then the curves $D_1$ and $D_2$ belong to $\mathcal{T}$; that is, each one is a product of polynomials defining curves in the pencil ${\mathcal P}_{\mathcal X}$ and this cannot happen since the curves  defined by $F_1, F_2, \ldots, F_r$ are components of the same curve of the pencil.

We conclude by proving the just alluded equality. $\mathcal{T}\subseteq \mathcal{L}(\nu{\bf r})$ is obvious. Now, reasoning by contradiction, assume that $\mathcal{T}\subsetneq \mathcal{L}(\nu{\bf r})$. The set $\Delta$ of generic elements in $\mathcal{L}(\nu{\bf r})$ which are not in  $\mathcal{T}$ is infinite because the generic elements of $\mathcal{L}(\nu{\bf r})$ are determined by the vectors in the complementary of a linear subvariety of $\mathbb{CP}^{s-1}$, where $s$ is the dimension of $\mathcal{L}(\nu{\bf r})$ (see Section \ref{bp}). Applying B\'ezout-Noether Formula (\ref{eq.Noether}) to any element $D\in \Delta$ and a generic element $G$ of the pencil ${\mathcal P}_{\mathcal X}$ we get
\begin{multline*}
\deg(D)\deg(G)-\sum_{P\in \tilde{D}\cap \tilde{G}} \mult_P(\tilde{D}) \mult_P(\tilde{G})\\
 \leq\deg(D)\deg(G)-\sum_{P\in \DX} \mult_P(\tilde{D}) \mult_P(\tilde{G})=\nu \bigg(n^2-\sum_{P\in \DX} r_P^2\bigg)=\nu \langle {\bf r},{\bf r}\rangle =0.
 \end{multline*}
This implies that $D\setminus \{\DX\cap \mathbb{CP}^2\}$ does not meet $G$. Since this happens for all generic element $G$ of  ${\mathcal P}_{\mathcal X}$, the irreducible components of $D$ must be irreducible components of non-generic elements of ${\mathcal P}_{\mathcal X}$. This is a contradiction because $\Delta$ is infinite and the set of non-generic curves in ${\mathcal P}_{\mathcal X}$ is finite. So $\mathcal{T}= \mathcal{L}(\nu{\bf r})$ and our proof is completed.
 \end{proof}

\begin{pro}\label{prop4}
The vector ${\bf r}$ generates the orthogonal complement of $S$ in $\mathbb{R}^{N+1}$ with respect to the bilinear form $\langle \cdot,\cdot \rangle$.
\end{pro}

\begin{proof}
Lemma \ref{independent} and the fact that $N$ is the cardinality of $S$ prove that the orthogonal complement of $S$ in $\mathbb{R}^{N+1}$ has dimension 1. Then the result follows from equalities (\ref{Eq.1}), (\ref{Eq.2}) and (\ref{Eq.3}).
\end{proof}

Next we state our main theorem, which justifies the forthcoming Corollary \ref{poin} and Algorithm \ref{alg1}. Corollary \ref{poin} states that the Poincar\'e problem can be solved for the family of vector fields $\mathbf{X}$ that admit a WAI polynomial first integral in the sense that the degree of the first integral can be obtained from the reduction of singularities of $\mathbf{X}$. Algorithm \ref{alg1} decides whether a vector field $\mathbf{X}$ has a WAI polynomial first integral or not, and computes a minimal one in the affirmative case.






\begin{teo}\label{char}
Let ${\bf X}$ be a planar polynomial vector field having a WAI polynomial first integral. Consider  its complex projectivization $\mathcal X$  and  the corresponding dicritical configuration $\DX$. Let $R_1, R_2\ldots, R_r$ be the maximal points of $\DX$ and set ${\rm Fr}(\DX)=\{P\in \DX\mid \mbox{ $P$ is free}\}$. Then the following statements hold:
\begin{itemize}
\item[(a)] The line at infinity is invariant by $\mathcal X$ and contains the points in $\DX \cap \mathbb{P}^2$.
\item[(b)] $R_1, R_2 \ldots,R_r$ are the unique infinitely near dicritical singularities of $\mathcal X$.
\item[(c)] The set ${\rm MFr}(\DX)$ of maximal elements in ${\rm Fr}(\DX)$ has cardinality $r$.
\item[(d)] Let ${\rm MFr}(\DX)=\{M_1, M_2, \ldots,M_r\}$. Then for each $i$, $1 \leq i \leq r$, there exists an invariant by $\mathcal X$ curve $C_i$ in the linear system ${\mathcal L}_{d_i}(\DX,  {\mathbf m}(\DX^{M_i},\DX))$, where ${\mathbf m}(\DX^{M_i},\DX):=(h^i_P)_{P \in \DX}$ and $d_i:=\sum_{P\in \DX\cap \tilde{L}} h^i_P$, such that $\mult_P(\tilde{C}_i)=h^i_P$.
\item[(e)] The set $S=\{{\bf c}_i\}_{i=1}^r\cup \{{\bf e}_Q\}_{Q\in \mathcal{N}(\mathcal{X})}\subseteq \mathbb{R}^{N+1}$ introduced in (\ref{ctoese}) is linearly independent.
\item[(f)] Let ${\bf R}=(n^{-};(r^{-}_P)_{P \in \DX})$ be the vector with non-negative integral components that generates the orthogonal complement, with respect to the bilinear pair $\langle \cdot,\cdot \rangle$ defined in (\ref{bilineal}), of the vector space that $S$ spans in $\mathbb{R}^{N+1}$ and such that $n^{-}>0$ and $\gcd(n^{-}; (r^{-}_P)_{P \in \DX})=1$. Then ${\bf R} ={\bf r}$, ${\bf r}=(n;(r_P)_{P \in \DX})$ being the vector defined in (\ref{elr}). Moreover, $n=\sum_{P\in{\mathcal D}({\mathcal X})\cap\tilde L}r_P$,  $n^2=\sum_{P \in \DX} r_P^2$ and there exist non-negative integers $n_i>0$,  $1\leq i\leq r$, and $b_P$, $P\in \mathcal{N}(\mathcal X)$, such that
\begin{equation}
\label{igual}
{\bf r}=\sum_{i=1}^r n_i{\bf c}_i+\sum_{P\in \mathcal{N}(\mathcal X)} b_P {\bf e}_P.
\end{equation}
\item[(g)] If $r\geq 2$ then, for each $i$ such that $1\leq i\leq r$, $C_i$ is the unique curve in the linear system ${\mathcal L}_{d_i}\left(\DX,  {\mathbf m}(\DX^{M_i},\DX)\right)$. If $r=1$ then $\mathbf{c}_1=\mathbf{r}$.
\item[(h)] Let $f_i(x,y)=0$ be an equation of the affine curve defined by $C_i$, $1\leq i\leq r$. Then,  $\prod_{i=1}^{r} f_i^{n_i}$ is a minimal WAI polynomial first integral of the vector field ${\bf X}$.

\end{itemize}
\end{teo}

\begin{proof}

Items (a)-(f), except Equality (\ref{igual}), follow  from the preceding paragraphs in this section. Notice that Proposition \ref{prop4} and Lemma \ref{bbb} prove the equality ${\bf R} ={\bf r}$ in Item (f). Let us show that Equality (\ref{igual}) holds.

Assume that $H=\prod_{i=1}^r f_i^{n_i}$ is a WAI polynomial first integral which, as usual, we pick minimal and let us prove the above equality. Consider the matrix $\mathbf{P}=(p_{P,Q})_{P,Q \in  \DX}$ such that $p_{P,Q}$ equals $-1$ (respectively, $1$, $0$) if $P=Q$ (respectively, $P$ is proximate to $Q$, otherwise). Then, by \cite[Theorem 4.5.2]{C}, the components of the vector $$(b_P)_{P \in  \DX}:=\mathbf{P}^{-1} \left(r_P-\mult_P(D)\right)_{P \in  \DX}$$ given by the global curve $D$ defined by $\prod_{i=1}^r F_i^{n_i}$, $F_i$ being the projectivization of $f_i$, are non-negative because  $D$ passes virtually through the cluster $(\DX, (r_P)_{P\in \DX})$. With the above information, set $\mathbf{w}$ the vector in $\mathbb{R}^{N+1}$ given by $$\mathbf{w}:=\sum_{i=1}^r n_i\mathbf{c}_i+\sum_{P \in \DX} b_P \mathbf{e}_P.$$

The equality $\mathbf{w}=\mathbf{r}$ holds because $\mathbf{P}$ is the change of basis matrix between the basis $\{\mathbf{e}_P\}_{P\in \DX}$ of  $\mathbb{R}^N$ and  the canonical one.  Finally, the equalities $\langle \mathbf{r}, \mathbf{r} \rangle=0$, $\langle \mathbf{r}, \mathbf{c}_i \rangle =0$, $1 \leq i \leq r$ and the inequalities $\langle \mathbf{r}, \mathbf{e}_P \rangle\geq 0$, $P\in \DX$, prove, by Part (d) of Proposition \ref{proposition2}, that $b_P=0$ whenever $P$ is an infinitely near dicritical singularity. This finishes the proof of Equality (\ref{igual}).

Now we prove Item (g). Reasoning as in the paragraph below (\ref{Eq.5}) one can show the inequalities $\langle \mathbf{c}_i, \mathbf{c}_i\rangle\leq 0$, $1 \leq i \leq r$, and also that  $\langle \mathbf{c}_i, \mathbf{c}_i\rangle= 0$ if and only if the vector $\mathbf{c}_i$ is a multiple of $\mathbf{r}$. In case $r\geq 2$, the vector $\mathbf{c}_i$ cannot be a multiple of $\mathbf{r}$ because $\mult_{R_j}(\tilde{C}_i)=0$ if $i\not=j$ and all the components of $\mathbf{r}$ are different from $0$. Then $\langle \mathbf{c}_i, \mathbf{c}_i\rangle<0$ and, therefore, $d_i^2<\sum_{P\in \DX} \mult_{P}(\tilde{C}_i)^2$. As a consequence, if $C_i$ is not the unique curve in the linear system ${\mathcal L}_{d_i}(\DX,  {\mathbf m}(\DX^{M_i},\DX))$, we get a contradiction by applying B\'ezout-Noether Formula for two generic curves of that system. Therefore  (g) is proved when $r\geq 2$. The result for $r=1$ holds by \cite[Theorem 1]{c-p-r-1}.

To conclude our proof, it only remains to show that Item (h) is true. Firstly, by  Lemma \ref{lema4}, $\mathcal{P}_{\mathcal X}=\mathcal{L}(\mathbf{r})$ . Now, on the one hand, the curve defined by $\prod_{i=1}^r F_i(X,Y,Z)^{n_i}$ belongs to the pencil ${\mathcal P}_{\mathcal{X}}=\mathcal{L}(\mathbf{r})$ in virtue of Equality (\ref{igual}) and Lemma \ref{lema3}. On the other hand, setting $\mathbf{l}=(1; (\mult_P(\tilde{L}))_{P\in \DX})$, we have $\langle \mathbf{r}, \mathbf{\mathbf{l}}\rangle =0$ by Equality (\ref{Eq.2}). So the non-reduced curve defined by $Z^n$ belongs also to the pencil $\mathcal{P}_{\mathcal X}$ by Lemma \ref{agh}. Therefore $\prod_{i=1}^r F_i(X,Y,Z)^{n_i}$ and $Z^n$ span the pencil and thus $H=\prod_{i=1}^r f_i^{n_i}$ is a WAI polynomial first integral.  Notice that $H$ is minimal  because, otherwise, $\gcd(n_1, n_2, \ldots,n_r)>1$, which contradicts the fact that the components of $\mathbf{r}$ have no common factor.

We finish by explaining that the curves $C_i$ have only one place at infinity. In fact, they have only one intersection point with the line at infinity (by items (d), (e) and B\'ezout-Noether Formula) and only one analytic branch at this point (by \cite[Theorem 3.5.3]{C}).
\end{proof}

\begin{cor}
\label{poin}
Let ${\bf X}$ be a planar polynomial vector field as in Theorem \ref{char}. Then:
\begin{enumerate}
\item The degree $n$ and the exponents $n_i$ of the (minimal) WAI polynomial first integral of ${\bf X}$ can be computed from the proximity graph of the dicritical configuration $\DX$ and the number of points in $\DX$ through which the strict transform of the infinity line passes.
    \item The proximity graph of $\DX$ determines a bound for the degree of the (minimal) WAI polynomial first integral.
\end{enumerate}
\end{cor}

\begin{proof}
Our first statement follows from items (f) and (h) of Theorem \ref{char} and Item (4) of Theorem \ref{prrr}. With respect to our second statement, it can be proved from the fact that the line at infinity only can go through some points in the first block of consecutive free points in $\DX$. So it suffices to consider the maximum of the degrees that can be computed as in Statement (1) for those finitely many possibilities.
\end{proof}

Next, we state the algorithm mentioned before Theorem \ref{char}, which will be followed by an example that explains how it works.
\medskip

\begin{alg}
\label{alg1}
{\rm

\begin{itemize} $\;$

\item \emph{Input:} An arbitrary polynomial vector field $\mathbf{X}$.

\item \emph{Output:} Either a minimal  WAI polynomial first integral of $\mathbf{X}$, or $0$ (in case $\mathbf{X}$ has no first integral of this type).

\end{itemize}
\medskip

\begin{enumerate}
 \item Compute the dicritical configuration ${\mathcal D}(\mathcal{X})$ of the complex projectivization $\mathcal{X}$ of $\mathbf{X}$. To do it, we need to perform the reduction of singularities of $\mathbf{X}$.
\item Let $r$ be the number of maximal points of ${\mathcal D}(\mathcal{X})$. If either ${\rm Fr}(\DX)$ has not $r$ maximal elements or Item (e) of Theorem \ref{char} is not satisfied, then return $0$.
\item Consider the linear systems defined in Item (g) of Theorem \ref{char} and compute an equation $f_i=0$ for the unique curve $C_i$, $1 \leq i \leq r$, there defined.
\item Compute the vector $\mathbf{R}$ in Item (f) of Theorem \ref{char}. If $\mathbf{R}$ does not satisfy the equalities in that item,  then return $0$. Let $K:=\prod_{i=1}^{r} f_i^{n_i}$ be the polynomial in Item (h) of Theorem \ref{char}, whose exponents are given by the vector $\mathbf{R}$. Check whether $K$ is a first integral of $\mathbf{X}$. If the answer is positive, then return $K$. Otherwise return $0$.

\end{enumerate}
}
\end{alg}

\begin{exa}
\label{elcuatro}
\begin{table}[h!]
\centering
    \begin{tabular}{||c|c||}
   \hline\hline Chart & System of coordinates \\\hline\hline
    $Y\not=0$  & $(x=X/Y, z=Z/Y)$ at $P_0$ \\
    $V_1^{P_0}$ & $(x_1=x, z_1=z/x)$ at $P_1$\\
    $V_2^{P_1}$ & $(x_2=x_1/z_1, z_2=z_1)$  at $P_2$\\
   $V_1^{P_2}$ & $(x_3=x_2, z_3=z_2/x_2-1)$ at $P_3$\\
$V_1^{P_{i-1}}$ & $(x_i=x_{i-1}, z_i=z_{i-1}/x_{i-1})$ at $P_i$, $4 \leq i \leq 13$\\
$V_1^{P_3}$ & $(x_{14}=x_3, z_{14}=z_3/x_3-1)$ at $P_{14}$\\
$V_1^{P_{i-1}}$ & $(x_i=x_{i-1}, z_i=z_{i-1}/x_{i-1}-1)$ at $P_i$, $15 \leq i \leq 23$\\
$V_1^{P_1}$ & $(x_{24}=x_1, z_{24}=z_1/x_1-1)$ at $P_{24}$ \\
$V_1^{P_{i-1}}$ & $(x_i=x_{i-1}, z_i=z_{i-1}/x_{i-1})$ at $P_i$, $25 \leq i \leq 28$ \\
$Z\not=0$; & $(x'=X/Z, y'=Y/Z)$  at $Q_0$ \\
$V_1^{Q_0}$ & $(x'_1=x', y'_1=y'/x')$ at $Q_1$ \\ \hline\hline
    \end{tabular}
\vspace{0.4cm}
    \caption{The configuration $\mathcal{S}(\mathcal{X})$.}
\label{tabla3}
\end{table}

Consider the polynomial vector field $\mathbf{X}$ defined by the following differential form:
$$(10 x^7 - 9 x^6  + 6 x^5 y  + 9 x^4 y  - 6 x^3 y  +
 6 x^2 y^2  + 2 x y^2 )dx+
(2 x^6  - x^4  + 6 x^3 y  - x^2 y  + 4 y^2)dy.$$
Taking projective coordinates $X,Y,Z$ and considering $x$ and $y$ as affine coordinates in the chart $Z\not=0$, $\mathbf{X}$ is extended to its complex projectivization $\mathcal X$ defined by the homogeneous 1-form $\omega=A\;dX+B\;dY+C\;dZ$, where
\[
\begin{split}
A=&10 X^7 Z - 9 X^6 Z^2 + 6 X^5 Y Z^2 + 9 X^4 Y Z^3 - 6 X^3 Y Z^4 + 6 X^2 Y^2 Z^4 + 2 X Y^2 Z^5,\\
B=&2 X^6 Z^2 - X^4 Z^4 + 6 X^3 Y Z^4 - X^2 Y Z^5 + 4 Y^2 Z^6,\\
C=& -10 X^8 + 9 X^7 Z - 8 X^6 Y Z - 9 X^5 Y Z^2 + 7 X^4 Y Z^3 \\
&- 12 X^3 Y^2 Z^3 - X^2 Y^2 Z^4 - 4 Y^3 Z^5.
\end{split}
\]
Applying the algorithm of reduction of singularities we obtain that the singular configuration of $\mathcal X$ is $\mathcal{S}(\mathcal{X})=\{P_i\}_{i=0}^{28}\cup \{Q_0,Q_1\}$, where the involved infinitely near points are those described in Table \ref{tabla3}.

The dicritical infinitely near singularities of $\mathcal{X}$ are $P_{13}, P_{23}$ and $P_{28}$. Therefore the configuration $\DX$ is $\{P_i\}_{i=0}^{28}$. We have depicted the proximity graph of this configuration in Figure \ref{F2}.

\smallskip

With the notations as above,  $r=3$, $R_1=M_1=P_{13}$, $R_2=M_2=P_{23}$ and $R_3=M_3=P_{28}$. Notice that $\DX\cap \tilde{L}=\{P_0,P_1\}$. The three first rows of the following matrix are, respectively, the vectors $\mathbf{c}_1, \mathbf{c}_2$ and $\mathbf{c}_3$, and the remaining ones are the vectors $\{\mathbf{e}_Q\}_{Q\in \mathcal{N}(\mathcal{X})}$:
{\tiny
$$
\arraycolsep=1pt
\left(
\begin{array}{cccccccccccccccccccccccccccccc}
 3 & 2 & 1 & 1 & 1 & 1 & 1 & 1 & 1 & 1 & 1 & 1 & 1 & 1 & 1 & 0 & 0 & 0 & 0 & 0 & 0 & 0 & 0 & 0 & 0 & 0 & 0 & 0 & 0 & 0 \\
 3 & 2 & 1 & 1 & 1 & 0 & 0 & 0 & 0 & 0 & 0 & 0 & 0 & 0 & 0 & 1 & 1 & 1 & 1 & 1 & 1 & 1 & 1 & 1 & 1 & 0 & 0 & 0 & 0 & 0 \\
 2 & 1 & 1 & 0 & 0 & 0 & 0 & 0 & 0 & 0 & 0 & 0 & 0 & 0 & 0 & 0 & 0 & 0 & 0 & 0 & 0 & 0 & 0 & 0 & 0 & 1 & 1 & 1 & 1 & 1 \\
 0 & -1 & 1 & 1 & 0 & 0 & 0 & 0 & 0 & 0 & 0 & 0 & 0 & 0 & 0 & 0 & 0 & 0 & 0 & 0 & 0 & 0 & 0 & 0 & 0 & 0 & 0 & 0 & 0 & 0 \\
 0 & 0 & -1 & 1 & 0 & 0 & 0 & 0 & 0 & 0 & 0 & 0 & 0 & 0 & 0 & 0 & 0 & 0 & 0 & 0 & 0 & 0 & 0 & 0 & 0 & 1 & 0 & 0 & 0 & 0 \\
 0 & 0 & 0 & -1 & 1 & 0 & 0 & 0 & 0 & 0 & 0 & 0 & 0 & 0 & 0 & 0 & 0 & 0 & 0 & 0 & 0 & 0 & 0 & 0 & 0 & 0 & 0 & 0 & 0 & 0 \\
 0 & 0 & 0 & 0 & -1 & 1 & 0 & 0 & 0 & 0 & 0 & 0 & 0 & 0 & 0 & 1 & 0 & 0 & 0 & 0 & 0 & 0 & 0 & 0 & 0 & 0 & 0 & 0 & 0 & 0 \\
 0 & 0 & 0 & 0 & 0 & -1 & 1 & 0 & 0 & 0 & 0 & 0 & 0 & 0 & 0 & 0 & 0 & 0 & 0 & 0 & 0 & 0 & 0 & 0 & 0 & 0 & 0 & 0 & 0 & 0 \\
 0 & 0 & 0 & 0 & 0 & 0 & -1 & 1 & 0 & 0 & 0 & 0 & 0 & 0 & 0 & 0 & 0 & 0 & 0 & 0 & 0 & 0 & 0 & 0 & 0 & 0 & 0 & 0 & 0 & 0 \\
 0 & 0 & 0 & 0 & 0 & 0 & 0 & -1 & 1 & 0 & 0 & 0 & 0 & 0 & 0 & 0 & 0 & 0 & 0 & 0 & 0 & 0 & 0 & 0 & 0 & 0 & 0 & 0 & 0 & 0 \\
 0 & 0 & 0 & 0 & 0 & 0 & 0 & 0 & -1 & 1 & 0 & 0 & 0 & 0 & 0 & 0 & 0 & 0 & 0 & 0 & 0 & 0 & 0 & 0 & 0 & 0 & 0 & 0 & 0 & 0 \\
 0 & 0 & 0 & 0 & 0 & 0 & 0 & 0 & 0 & -1 & 1 & 0 & 0 & 0 & 0 & 0 & 0 & 0 & 0 & 0 & 0 & 0 & 0 & 0 & 0 & 0 & 0 & 0 & 0 & 0 \\
 0 & 0 & 0 & 0 & 0 & 0 & 0 & 0 & 0 & 0 & -1 & 1 & 0 & 0 & 0 & 0 & 0 & 0 & 0 & 0 & 0 & 0 & 0 & 0 & 0 & 0 & 0 & 0 & 0 & 0 \\
 0 & 0 & 0 & 0 & 0 & 0 & 0 & 0 & 0 & 0 & 0 & -1 & 1 & 0 & 0 & 0 & 0 & 0 & 0 & 0 & 0 & 0 & 0 & 0 & 0 & 0 & 0 & 0 & 0 & 0 \\
 0 & 0 & 0 & 0 & 0 & 0 & 0 & 0 & 0 & 0 & 0 & 0 & -1 & 1 & 0 & 0 & 0 & 0 & 0 & 0 & 0 & 0 & 0 & 0 & 0 & 0 & 0 & 0 & 0 & 0 \\
 0 & 0 & 0 & 0 & 0 & 0 & 0 & 0 & 0 & 0 & 0 & 0 & 0 & -1 & 1 & 0 & 0 & 0 & 0 & 0 & 0 & 0 & 0 & 0 & 0 & 0 & 0 & 0 & 0 & 0 \\
 0 & 0 & 0 & 0 & 0 & 0 & 0 & 0 & 0 & 0 & 0 & 0 & 0 & 0 & 0 & -1 & 1 & 0 & 0 & 0 & 0 & 0 & 0 & 0 & 0 & 0 & 0 & 0 & 0 & 0 \\
 0 & 0 & 0 & 0 & 0 & 0 & 0 & 0 & 0 & 0 & 0 & 0 & 0 & 0 & 0 & 0 & -1 & 1 & 0 & 0 & 0 & 0 & 0 & 0 & 0 & 0 & 0 & 0 & 0 & 0 \\
 0 & 0 & 0 & 0 & 0 & 0 & 0 & 0 & 0 & 0 & 0 & 0 & 0 & 0 & 0 & 0 & 0 & -1 & 1 & 0 & 0 & 0 & 0 & 0 & 0 & 0 & 0 & 0 & 0 & 0 \\
 0 & 0 & 0 & 0 & 0 & 0 & 0 & 0 & 0 & 0 & 0 & 0 & 0 & 0 & 0 & 0 & 0 & 0 & -1 & 1 & 0 & 0 & 0 & 0 & 0 & 0 & 0 & 0 & 0 & 0 \\
 0 & 0 & 0 & 0 & 0 & 0 & 0 & 0 & 0 & 0 & 0 & 0 & 0 & 0 & 0 & 0 & 0 & 0 & 0 & -1 & 1 & 0 & 0 & 0 & 0 & 0 & 0 & 0 & 0 & 0 \\
 0 & 0 & 0 & 0 & 0 & 0 & 0 & 0 & 0 & 0 & 0 & 0 & 0 & 0 & 0 & 0 & 0 & 0 & 0 & 0 & -1 & 1 & 0 & 0 & 0 & 0 & 0 & 0 & 0 & 0 \\
 0 & 0 & 0 & 0 & 0 & 0 & 0 & 0 & 0 & 0 & 0 & 0 & 0 & 0 & 0 & 0 & 0 & 0 & 0 & 0 & 0 & -1 & 1 & 0 & 0 & 0 & 0 & 0 & 0 & 0 \\
 0 & 0 & 0 & 0 & 0 & 0 & 0 & 0 & 0 & 0 & 0 & 0 & 0 & 0 & 0 & 0 & 0 & 0 & 0 & 0 & 0 & 0 & -1 & 1 & 0 & 0 & 0 & 0 & 0 & 0 \\
 0 & 0 & 0 & 0 & 0 & 0 & 0 & 0 & 0 & 0 & 0 & 0 & 0 & 0 & 0 & 0 & 0 & 0 & 0 & 0 & 0 & 0 & 0 & -1 & 1 & 0 & 0 & 0 & 0 & 0 \\
 0 & 0 & 0 & 0 & 0 & 0 & 0 & 0 & 0 & 0 & 0 & 0 & 0 & 0 & 0 & 0 & 0 & 0 & 0 & 0 & 0 & 0 & 0 & 0 & 0 & -1 & 1 & 0 & 0 & 0 \\
 0 & 0 & 0 & 0 & 0 & 0 & 0 & 0 & 0 & 0 & 0 & 0 & 0 & 0 & 0 & 0 & 0 & 0 & 0 & 0 & 0 & 0 & 0 & 0 & 0 & 0 & -1 & 1 & 0 & 0 \\
 0 & 0 & 0 & 0 & 0 & 0 & 0 & 0 & 0 & 0 & 0 & 0 & 0 & 0 & 0 & 0 & 0 & 0 & 0 & 0 & 0 & 0 & 0 & 0 & 0 & 0 & 0 & -1 & 1 & 0 \\
 0 & 0 & 0 & 0 & 0 & 0 & 0 & 0 & 0 & 0 & 0 & 0 & 0 & 0 & 0 & 0 & 0 & 0 & 0 & 0 & 0 & 0 & 0 & 0 & 0 & 0 & 0 & 0 & -1 & 1
\end{array}
\right).$$
}

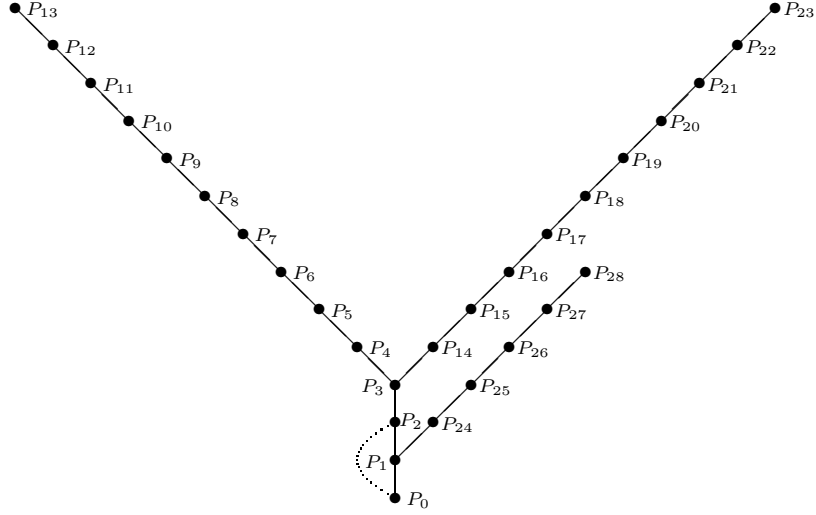
\begin{figure}[ht!]
\centering
\setlength{\unitlength}{0.5cm}
\begin{picture}(23,14)
\qbezier[20](11,0)(9,1)(11,2)

\put(11,0){\circle*{0.3}}\put(11.3,-0.2){\tiny{$P_0$}}
\put(11,0){\line(0,1){1}}

\put(11,1){\circle*{0.3}}\put(10.2,0.8){\tiny{$P_1$}}
\put(11,1){\line(0,1){1}}

\put(11,2){\circle*{0.3}}\put(11.1,1.9){\tiny{$P_2$}}
\put(11,2){\line(0,1){1}}

\put(11,3){\circle*{0.3}}\put(10.1,2.8){\tiny{$P_3$}}
\put(11,3){\line(-1,1){1}}

\put(10,4){\circle*{0.3}}\put(10.3,3.8){\tiny{$P_4$}}
\put(10,4){\line(-1,1){1}}

\put(9,5){\circle*{0.3}}\put(9.3,4.8){\tiny{$P_5$}}
\put(9,5){\line(-1,1){1}}

\put(8,6){\circle*{0.3}}\put(8.3,5.8){\tiny{$P_6$}}
\put(8,6){\line(-1,1){1}}

\put(7,7){\circle*{0.3}}\put(7.3,6.8){\tiny{$P_7$}}
\put(7,7){\line(-1,1){1}}

\put(6,8){\circle*{0.3}}\put(6.3,7.8){\tiny{$P_8$}}
\put(6,8){\line(-1,1){1}}

\put(5,9){\circle*{0.3}}\put(5.3,8.8){\tiny{$P_9$}}
\put(5,9){\line(-1,1){1}}

\put(4,10){\circle*{0.3}}\put(4.3,9.8){\tiny{$P_{10}$}}
\put(4,10){\line(-1,1){1}}

\put(3,11){\circle*{0.3}}\put(3.3,10.8){\tiny{$P_{11}$}}
\put(3,11){\line(-1,1){1}}

\put(2,12){\circle*{0.3}}\put(2.3,11.8){\tiny{$P_{12}$}}
\put(2,12){\line(-1,1){1}}

\put(1,13){\circle*{0.3}}\put(1.3,12.8){\tiny{$P_{13}$}}


\put(11,1){\line(1,1){1}}

\put(12,2){\circle*{0.3}}\put(12.2,1.8){\tiny{$P_{24}$}}
\put(12,2){\line(1,1){1}}

\put(13,3){\circle*{0.3}}\put(13.2,2.8){\tiny{$P_{25}$}}
\put(13,3){\line(1,1){1}}

\put(14,4){\circle*{0.3}}\put(14.2,3.8){\tiny{$P_{26}$}}
\put(14,4){\line(1,1){1}}

\put(15,5){\circle*{0.3}}\put(15.2,4.8){\tiny{$P_{27}$}}
\put(15,5){\line(1,1){1}}

\put(16,6){\circle*{0.3}}\put(16.2,5.8){\tiny{$P_{28}$}}


\put(11,3){\line(1,1){1}}

\put(12,4){\circle*{0.3}}\put(12.2,3.8){\tiny{$P_{14}$}}
\put(12,4){\line(1,1){1}}

\put(13,5){\circle*{0.3}}\put(13.2,4.8){\tiny{$P_{15}$}}
\put(13,5){\line(1,1){1}}

\put(14,6){\circle*{0.3}}\put(14.2,5.8){\tiny{$P_{16}$}}
\put(14,6){\line(1,1){1}}

\put(15,7){\circle*{0.3}}\put(15.2,6.8){\tiny{$P_{17}$}}
\put(15,7){\line(1,1){1}}

\put(16,8){\circle*{0.3}}\put(16.2,7.8){\tiny{$P_{18}$}}
\put(16,8){\line(1,1){1}}

\put(17,9){\circle*{0.3}}\put(17.2,8.8){\tiny{$P_{19}$}}
\put(17,9){\line(1,1){1}}

\put(18,10){\circle*{0.3}}\put(18.2,9.8){\tiny{$P_{20}$}}
\put(18,10){\line(1,1){1}}

\put(19,11){\circle*{0.3}}\put(19.2,10.8){\tiny{$P_{21}$}}
\put(19,11){\line(1,1){1}}

\put(20,12){\circle*{0.3}}\put(20.2,11.8){\tiny{$P_{22}$}}
\put(20,12){\line(1,1){1}}

\put(21,13){\circle*{0.3}}\put(21.2,12.8){\tiny{$P_{23}$}}
\end{picture}

\caption{Proximity graph of $\DX$.}
\label{F2}
\end{figure}

The set $S=\{\mathbf{c}_1,\mathbf{c}_2,\mathbf{c}_3\}\cup \{\mathbf{e}_Q\}_{Q\in \mathcal{N}(\mathcal{X})}$ is linearly independent and the orthogonal complement, with respect to the bilinear pairing $\langle \cdot, \cdot\rangle$, of the linear space that $S$ spans  is generated by
$$ \mathbf{R}=(10; 6,4,2,2,1,1,1,1,1,1,1,1,1,1,1,1,1,1,1,1,1,1,1,1,2,2,2,2,2).$$

Applying Corollary \ref{poin}, if  $\mathbf{X}$ has a WAI polynomial first integral, then its degree (that of a minimal one) is $10$.  Moreover the linear system $\mathcal{L}(\mathbf{c}_1)$  (respectively, $\mathcal{L}(\mathbf{c}_2)$,  $\mathcal{L}(\mathbf{c}_3)$)  has a unique curve (that is generic for the linear system): that defined by the equation $X^3-X^2Z+YZ^2=0$ (respectively, $  X^3+YZ^2=0$, $ X^2+YZ=0$).

Is is straightforward to check that $\mathbf{R}$ satisfies the two equalities above equation \eqref{igual} in Item (f) of Theorem \ref{char} and moreover that
$$ \mathbf{R} = \mathbf{c}_1+\mathbf{c}_2+2\mathbf{c}_3.$$
Therefore, items (d)-(f) of Theorem \ref{char} are satisfied.

The polynomial $K$ in  Step (4) of Algorithm \ref{alg1} is $K=(y-x^2+x^3)(y+x^3)(x^2+y)^2$. It is straightforward to check that this polynomial is a WAI  minimal  first integral of $\mathbf{X}$.
\end{exa}

\subsection{A classical alternative to Step (4) of Algorithm \ref{alg1}}

As mentioned in the introduction of this paper, Darboux proved in \cite{dar} that if a polynomial vector field $\mathbf{X}$ (of degree $d$) has at least $\binom{d + 1}2+1$ invariant algebraic curves, then it has a (Darboux) first integral, which can be computed using these invariant algebraic curves. In addition Jouanolou proved in \cite{jou} that if that number is at least $\binom{d + 1}2 + 2$, then the system has a rational first integral. These results were improved in \cite{CL2} (see also \cite{Ch,CLS,Che}). Next we state  Darboux and Jouanolou results adapted to our purposes.

\begin{teo}\label{tDar}
Suppose that a polynomial system $\mathbf{X}$ as in \eqref{e1} of degree $d$ admits $r$ irreducible invariant algebraic curves $f_i(x,y)=0$ with respective cofactor $k_i(x,y)$, $1 \leq i \le r$. Then:
\begin{enumerate}
\item[{\rm(a)}] There exist $\lambda_i\in \mathbb{C}$, not all zero, such that
    \begin{equation}
    \label{gege}
    \mathop{\sum}\limits_{i=1}^r\lambda_i k_i(x,y)=0
   \end{equation}
     if and only if the function
\begin{equation}\label{idar}
H=f_1^{\la_1}\cdots f_p^{\la_r}
\end{equation}
is a first integral of the system $\mathbf{X}$.

\item[\rm(b)] If $r=\binom {d+1}2+1$, then there exist $\lambda_i\in\C$, not all zero, such that $\mathop{\sum}\limits_{i=1}^r\lambda_i k_i(x,y)=0$.

\item[\rm(c)] If $r \geq \binom {d+1}2+2$, then $\mathbf{X}$ has a rational first integral.
\end{enumerate}
\end{teo}

Finding invariant algebraic curves is an important tool in the study of Darboux integrability and a very hard problem. Steps (1)-(3) of Algorithm \ref{alg1} provide $r$ candidates to be invariant curves of $\mathbf{X}$  given by equations $f_i=0$, $1 \leq i \leq r$. Thus, these curves are candidates to determine a Darboux first integral \eqref{idar}. After computing their cofactors
\begin{equation}
\label{cofac}
k_i(x,y)=\frac{P\pd {f_i}x+Q\pd{f_i}y}{f_i},
\end{equation}
we can  check whether there exist values $\lambda_i\in\N\cup\{0\}$ satisfying Equality (\ref{gege}), since we only need to solve a homogeneous linear system of equations. We notice that this linear system has  $\binom {d+1}2$ equations, corresponding with the number of monomials of a polynomial of degree $d-1$ in two variables, and $r$ unknowns, say the $\lambda_i$. If such values $\lambda_i$ exist, then we have succeeded and \eqref{idar} is a first integral of the system $\mathbf{X}$. Otherwise the first integral we are looking for does not exist.

As a consequence, we have designed an alternative algorithm to Algorithm \ref{alg1}. It has the same input an output and the same steps (1)-(3).

\begin{alg}
\label{alg2}
{\rm

\begin{itemize} $\;$

\item Input, Output and Steps (1)-(3) as in Algorithm \ref{alg1}.

\end{itemize}

\begin{enumerate}

 \item[{\rm(4)}] Compute the cofactors $k_i(x,y)$ corresponding to the curves $f_i=0$, $1 \leq i \leq r$, as in (\ref{cofac}).

 \item[{\rm(5)}] Solve the homogeneous complex linear system of equations $\mathop{\sum}\limits_{i=1}^r\lambda_i k_i(x,y)=0$, where the unknowns are $\lambda_i$, $1 \leq i \leq r$. If it has a solution $\lambda_i = n_i \in \mathbb{N}$, $\gcd(n_1,n_2, \ldots,n_r)=1$, then return $K= \prod_{i=1}^{r} f_i^{n_i}$. Otherwise return $0$.

\end{enumerate}
}
\end{alg}

\begin{exa}
{\rm
Consider the vector field $\mathbf{X}$ in Example \ref{elcuatro} and the polynomial invariant curves $f_1 = y-x^2+x^3$, $f_2=y+x^3$ and $f_3=x^2+y$ there computed by using steps (1)-(3) of Algorithm \ref{alg1}. Now Step (4) of Algorithm \ref{alg2} determines the cofactors: $k_1= 2x(-x^2-4x^3+3x^4-5y+3xy) $, $k_2= 2x(3x^2-5x^3+3x^4-y+3xy)$ and $k_3=x(-2x^2+9x^3-6x^4+6y-6xy) $. Finally, solving the linear system in Step (5) of Algorithm \ref{alg2}, we get $n_1=n_2=1$ and $n_3=2$,  which are coprime and provide a minimal WAI polynomial first integral of $\mathbf{X}$.
}
\end{exa}

\end{document}